\documentclass[final,12pt,a4paper,twosided]{amsart} 

\usepackage{tikz,tikz-cd} 
\usetikzlibrary{shapes.geometric,positioning}
\usetikzlibrary{arrows,decorations.pathmorphing,decorations.pathreplacing}
\usetikzlibrary{matrix,arrows.meta}
\usepackage{tkz-graph}
 \usepackage[all]{xy}
\usepackage{graphicx}

\usetikzlibrary{calc,decorations.pathmorphing,shapes}

\newcounter{sarrow}

\makeatletter
\def\nodedistance{\tikz@node@distance}
\makeatother

\usepackage{graphicx}
\usepackage[all]{xy}
\usepackage{placeins}
\usepackage{enumitem}
\usepackage{amssymb}
\usepackage{latexsym}
\usepackage{amsmath}
\usepackage{mathrsfs}
\usepackage{array,booktabs}
\usepackage{verbatim}
\usepackage{fullpage}
\usepackage[notref, notcite]{showkeys}  
\usepackage{color}
\usepackage[normalem]{ulem}
\usepackage{float}

\usepackage{soul}

\usepackage{color}
\definecolor{teal}{rgb}{0.0, 0.5, 0.5}
\definecolor{tealblue}{rgb}{0.21, 0.46, 0.53}
\definecolor{tealgreen}{rgb}{0.0, 0.51, 0.5}
\definecolor{tuscanred}{rgb}{0.51, 0.21, 0.21}
\definecolor{sangria}{rgb}{0.57, 0.0, 0.04}
\definecolor{rufous}{rgb}{0.66, 0.11, 0.03}
\definecolor{pinegreen}{rgb}{0.0, 0.47, 0.44}
\definecolor{darkscarlet}{rgb}{0.34, 0.01, 0.1}
\definecolor{darkseagreen}{rgb}{0.56, 0.74, 0.56}
\definecolor{darkpastelred}{rgb}{0.76, 0.23, 0.13}
\definecolor{darkpink}{rgb}{0.91, 0.33, 0.5}
\definecolor{darkpastelblue}{rgb}{0.47, 0.62, 0.8}
\definecolor{alizarin}{rgb}{0.82, 0.1, 0.26}
\definecolor{candyapplered}{rgb}{1.0, 0.03, 0.0}

\newcommand{\Pm}{P_{min}}

\DeclareMathOperator{\match}{Match \, }

\usepackage{hyperref}

\newcommand{\hyref}[2]{ \hyperref[#2]{#1~\ref*{#2}} }

\newcommand{\Canakci}{\c{C}anak\c{c}\i}
\newcommand{\Ilke}{\.{I}lke }

\theoremstyle{plain}
\newtheorem{theorem}{Theorem}[section]

\newtheorem{lemma}[theorem]{Lemma}
\newtheorem{corollary}[theorem]{Corollary}
\newtheorem{proposition}[theorem]{Proposition}

\theoremstyle{definition}
\newtheorem{remark}[theorem]{Remark}

\newtheorem{example}[theorem]{Example}
\newtheorem{definition}[theorem]{Definition}

\DeclareMathAlphabet{\mathpzc}{OT1}{pzc}{m}{it}

\newcommand{\calb}{\mathcal{B}}

\newcommand{\calg}{\mathcal{G}}

\newcommand{\calh}{\mathcal{H}}

\newcommand{\call}{\mathcal{L}}

\newcommand{\cals}{\mathcal{S}}

\renewcommand{\setminus}{\backslash}
\renewcommand{\emptyset}{\varnothing}

\newcommand{\arr}{\ar@{-}[r]}

\renewcommand{\phi}{\varphi}
\renewcommand{\epsilon}{\varepsilon}

\begin{document}

\title[Lattice bijections for submodules]{Lattice bijections for string modules, snake graphs and the weak Bruhat order}

\thanks{This work was supported through the Engineering and Physical Sciences Research Council, grant number EP/K026364/1, UK. The first author was also supported by EPSRC through EP/P016014/1 and the second author by the EPSRC through an Early Career Fellowship EP/P016294/1.}

\author{\Ilke \Canakci}
\address{School of Mathematics, Statistics and Physics, Newcastle University, Newcastle Upon Tyne NE1 7RU, United Kingdom}
\email{ilke.canakci@ncl.ac.uk}
\author{Sibylle Schroll} 
\address{Department of Mathematics, University of Leicester, University Road, Leicester LE1 7RH, United Kingdom}
\email{schroll@leicester.ac.uk}

\begin{abstract}
In this paper we introduce abstract string modules and give an explicit bijection between the submodule lattice of an abstract string module and the perfect matching lattice of the corresponding  abstract snake graph. In particular, we make explicit the direct correspondence between a submodule of a string module and the perfect matching of the corresponding snake graph. For every string module, we define a Coxeter element in a symmetric group, and we establish a bijection between these lattices and the interval in the weak Bruhat order determined by the Coxeter element. Using the correspondence between string modules and snake graphs, we give a new concise formulation of snake graph calculus.
\end{abstract}

\date{\today}

\maketitle



\section{Introduction}

Recently there have been many results relating the weak Bruhat order of a Coxeter group to representation theoretic notions such as  $\tau$-tilting theory, maximal green sequences  and torsion classes, see for example \cite{BCZ, Engenhorst, IRRT, Mizuno}.  In this paper, we establish an explicit combinatorial bijection between an interval in the weak Bruhat order of a symmetric group, the submodule lattice of a string module and the perfect matching lattice of a snake graph. In the process, we show how to determine the perfect matching corresponding to a submodule of a  given  submodule of a string module $M$ corresponding to a snake graph $\calg$ and conversely how to determine the submodule corresponding to a given perfect matching.

Snake graphs have been introduced in \cite{MSW1}  and used in \cite{MSW2} to give a combinatorial description of a basis of cluster algebras related to oriented surfaces. Since their introduction, snake graphs have 
been related to many different subjects such as number theory and knot theory. For the latter, in a recent paper \cite{LS} they have been used to  realise Jones polynomials of $2$-bridge knots as specialisations of cluster variables.  For the former, the introduction of the notion of abstract snake graph and the snake graph calculus developed in \cite{CS1,CS2,CS3} give a snake graph interpretation of skein relations and this was used to establish a bijection between snake graphs and positive finite continued fractions \cite{CS4}. 
In Section~\ref{Sec::NewBijection} of this paper, in Theorem~\ref{thm::CS bijection}, we give a new very concise formulation of the snake graph calculus in \cite{CS1}, significantly condensing the statement and thus making it more accessible and easier to work with, by introducing a uniform local configuration for the bijection rather than the step-by-step process given in \cite{CS1}. We also note that the bijection induces (possibly split) extensions between the corresponding submodules, see Proposition~\ref{prop::extension}.

In the representation theory of algebras, one of the best understood type of modules are the so-called \emph{string modules}. String modules can be defined using word combinatorics, and many of the representation theoretic properties can be deduced from the word combinatorics. 

The main motivation of this paper is to give a series of explicit bijections between perfect matching lattices of snake graphs, canonical submodule lattices of string modules (as defined in Section~\ref{def::canonical submodule lattice}) and intervals in the weak Bruhat order of symmetric groups. More precisely, based on the close connection of snake graphs and string modules,  we introduce the notion of an abstract string module through abstract word combinatorics on a two letter alphabet. The idea is that an abstract string module is like a generic string module, which when we specify an algebra might specialise to a string module (of the same form) over that algebra.  We then show how to obtain the perfect matching of an abstract snake graph corresponding to a submodule of an abstract string module and vice versa. We then  show that the lattice of (canonically embedded) submodules  of an abstract string module is in bijection with the perfect matching lattice of the corresponding abstract snake graph by giving an explicit bijection between  a perfect matching of a snake graph and the corresponding submodule. That such a bijection exists follows from  the fact that in terms of perfect matchings of snake graphs the cluster character  \cite{CC, P} is  the same as the expansion formula for cluster variables  \cite{MSW1} corresponding to arcs in a surface  and a remark to this effect already appears in \cite{MSW2}.

Given a string module $M$, in Section~\ref{sec::weak Bruhat} we define an associated Coxeter element $\sigma_M$ in a symmetric group $\frak S_M$ determined by $M$. Denote the  canonical submodule lattice of $M$ by  $\call(M)$ and the lattice of perfect matchings  of the snake graph $\calg$ associated to $M$ by $\call (\calg)$ (see Section~\ref{sec::abstract snake graphs}). The main result of this paper is the following.
 
\textbf{{ Theorem}} (Theorem~\ref{thm::LatticeBijections}). {\it
Let $M$ be a string module over a finite dimensional algebra $A=KQ/I$ with corresponding snake graph $\calg$ and let $\sigma_M$ be the associated Coxeter element. Then we have  bijections between 
\begin{enumerate}
\item the lattice corresponding to the interval $[e,\sigma_M]$ in the weak Bruhat order of $\frak S_M$, 
\item  the  canonical submodule lattice  $\call (M)$, 
\item the lattice of perfect matchings $\call (\calg)$.
\end{enumerate}
}

{\it Acknowledgements:} We would like to thank Nathan Reading for helpful discussions on distributive lattices.


\section{Submodules and perfect matchings of snake graphs}

\subsection{Abstract snake graphs}\label{sec::abstract snake graphs}

\begin{definition}\label{def:abstract snake graph}  1) We define a \textit{tile}  to be a graph with four vertices and four edges such that each vertex lies in exactly two edges. A geometric realisation of this graph is a square embedded in the plane (such that the edges are parallel to the $x$ and $y$ axes)  and we freely refer to the geometric realisation as a tile.

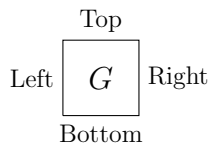
\begin{figure}[!htbp]
\begin{tikzpicture}[scale=1]
\draw (0,0)-- node[below, scale=.8] {Bottom} (1,0)-- node[right, scale=.8] {Right}(1,1)-- node[above, scale=.8] {Top}(0,1)--node[left, scale=.8] {Left} (0,0);
\node[scale=1] at (.5,.5){$G$};
\end{tikzpicture}  
\end{figure}

2) We say that two tiles are \textit{glued} if they share a common edge.  

3) Starting with a tile $T_1$, we glue a tile $T_2$ either to the right or top edge of $T_1$. Then glue $T_3$ either to the right or top edge of $T_2$ and so on up to gluing a tile $T_n$ to the right or top edge of $T_{n-1}$. 
We call $T_1, \ldots, T_n$ together with a gluing of edges a \textit{snake graph} $\calg$.  
 \end{definition}

\begin{remark}
This definition is different from the original definition in \cite{MS} and does not give the edge weights of the snake graph. The definition is, however, very similar to the definition of abstract snake graphs in \cite{CS1}.
\end{remark}

  \begin{definition} 
  1) We call a snake graph $\calg$ with tiles $T_1, \ldots, T_n$  a \textit{zigzag} if  for all $i$ either $T_i$ is glued on top of $T_{i-1}$ and $T_{i+1}$ is glued to the right of $T_i$ or if $T_i$ is glued to the right of $T_{i-1}$ and $T_{i+1}$ is glued on top of $T_i$. 
  
  2) We call a snake graph  $\calg$ with tiles $T_1, \ldots, T_n$  a \textit{straight piece} if   for all $i$ either $T_i$ is glued on to the right of $T_{i-1}$ and $T_{i+1}$ is glued to the right of $T_i$ or if $T_i$ is glued on top of $T_{i-1}$ and $T_{i+1}$ is glued on top of $T_i$. 
  \end{definition}

\begin{definition}
Given a finite set $S$, we label each tile of a snake graph $\calg$ by an element in $S$ and we call $S$ a set of \textit{face weights} of $\calg$.
\end{definition}  
Note that two tiles of a snake graph might have the same face weight.   
 
 \begin{definition} 
 A \textit{perfect matching} $P$ of a graph $G$ is a subset of the edges of $G$ such that every vertex of $G$ is incident to exactly one edge in $P$. We denote by 
$\match (G) $ the set of perfect matchings of $G$. 

For $H$ a subgraph of $G$, we denote by $P\vert_H$ the restriction of $P$ to $H$.  We note that $P\vert_H$ may not be a perfect matching of $H$.
\end{definition} 
  
For a snake graph $\calg$, the set of perfect matchings of $\calg$ forms a distributive lattice $\call (\calg)$ \cite{MSW1}.  The edges in $\call (\calg)$ can be labelled by the   face weights of $\calg$ in the following way. Two perfect matchings in $\call(\calg)$ are connected by an edge if the corresponding matchings are obtained by rotating a tile in $\calg$ and we label the corresponding edge in the Hasse diagram of $\call(\calg)$ by the face weight of this tile. 

\begin{example}
Example of a perfect matching lattice with edges labelled by face weights.

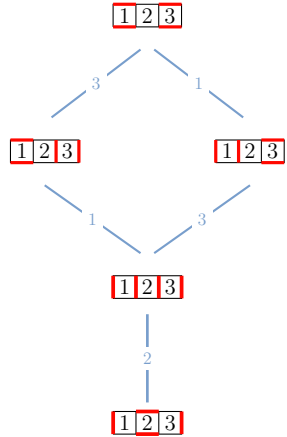
\[
\begin{tikzpicture}[scale=.9,rotate=180]
\node at (4.5,8){$\begin{tikzpicture} [scale=.3]
\draw (1,1)--(4,1)--(4,2)--(1,2)--(1,1);
\draw (2,1)--(2,2);
\draw (3,1)--(3,2);
\draw[color=candyapplered, line width=1.4] (1,1)--(1,2) (2,2)--(3,2) (2,1)--(3,1) (4,1)--(4,2);
\node[scale=.7] at (1.5, 1.5){$1$};
\node[scale=.7] at (2.5, 1.5){$2$};
\node[scale=.7] at (3.5, 1.5){$3$};
\end{tikzpicture}$};

\node at (4.5,6){$\begin{tikzpicture} [scale=.3]
\draw (1,1)--(4,1)--(4,2)--(1,2)--(1,1);
\draw (2,1)--(2,2);
\draw (3,1)--(3,2);
\draw[color=candyapplered, line width=1.4] (1,1)--(1,2) (2,1)--(2,2) (3,1)--(3,2) (4,1)--(4,2);
\node[scale=.7] at (1.5, 1.5){$1$};
\node[scale=.7] at (2.5, 1.5){$2$};
\node[scale=.7] at (3.5, 1.5){$3$};
\end{tikzpicture}$};

\node at (6,4){$\begin{tikzpicture} [scale=.3]
\draw (1,1)--(4,1)--(4,2)--(1,2)--(1,1);
\draw (2,1)--(2,2);
\draw (3,1)--(3,2);
\draw[color=candyapplered, line width=1.2] (1,1)--(2,1) (1,2)--(2,2) (3,1)--(3,2) (4,1)--(4,2);
\node[scale=.7] at (1.5, 1.5){$1$};
\node[scale=.7] at (2.5, 1.5){$2$};
\node[scale=.7] at (3.5, 1.5){$3$};
\end{tikzpicture}$};

\node at (3,4){$\begin{tikzpicture} [scale=.3]
\draw (1,1)--(4,1)--(4,2)--(1,2)--(1,1);
\draw (2,1)--(2,2);
\draw (3,1)--(3,2);
\draw[color=candyapplered, line width=1.2] (1,1)--(1,2) (2,1)--(2,2) (3,1)--(4,1) (3,2)--(4,2);
\node[scale=.7] at (1.5, 1.5){$1$};
\node[scale=.7] at (2.5, 1.5){$2$};
\node[scale=.7] at (3.5, 1.5){$3$};
\end{tikzpicture}$};

\node at (4.5, 2){$\begin{tikzpicture} [scale=.3]
\draw (1,1)--(4,1)--(4,2)--(1,2)--(1,1);
\draw (2,1)--(2,2);
\draw (3,1)--(3,2);
\draw[color=candyapplered, line width=1.2] (1,1)--(2,1) (1,2)--(2,2) (3,1)--(4,1) (3,2)--(4,2);
\node[scale=.7] at (1.5, 1.5){$1$};
\node[scale=.7] at (2.5, 1.5){$2$};
\node[scale=.7] at (3.5, 1.5){$3$};
\end{tikzpicture}$};

\path[color=darkpastelblue, line width=1] (4.5, 6.4) edge node [pos=.5, fill=white,outer sep=1mm,scale=.55]{$2$}  (4.5, 7.7);

\path[color=darkpastelblue, line width=.8] (4.4, 5.5) edge node [pos=.5, fill=white,outer sep=1mm,scale=.55]{$3$}  (3, 4.5);
\path[color=darkpastelblue, line width=.8] (4.6, 5.5) edge node [pos=.5, fill=white,outer sep=1mm,scale=.55]{$1$}  (6, 4.5);

\path[color=darkpastelblue, line width=.8] (3.1, 3.5) edge node [pos=.5, fill=white,outer sep=1mm,scale=.55]{$1$}  (4.4, 2.5);
\path[color=darkpastelblue, line width=.8] (5.9, 3.5) edge node [pos=.5, fill=white,outer sep=1mm,scale=.55]{$3$}  (4.6, 2.5);

\end{tikzpicture}
\]

\end{example}

\subsection{Abstract strings }
 
Let $\{ \to, \leftarrow\}$ be a set of  two letters where we refer to the first as a direct arrow, the second as an inverse arrow. An abstract string is a finite word in this alphabet or it is the additional word denoted by $\emptyset$. 

We call an abstract string consisting of only direct (resp. inverse) arrows a \emph{direct (resp. inverse) string. }

  We will see later in the paper that the correspondence between abstract  strings and abstract snake graphs as defined below is given by the arrow function defined in Section~\ref{subsec:string module of snake graph}.

\subsection{Construction of the snake graph of an abstract string}\label{Construction snake graph of w} 
Let $w = a_1 \ldots a_n $ be an abstract string with $a_i \in \{ \to, \leftarrow \}$ or $w =  \emptyset$. We iteratively construct a snake graph with $n+1$ tiles in the following way:
If $w = \emptyset$ then the corresponding abstract snake graph is given by a single tile. Otherwise if there is at least one letter, then $a_1 \ldots a_n$ is a concatenation of a collection of alternating  maximal direct and inverse strings $w_i$ such that $w = w_1 \ldots w_k$. We note that each $w_i$ might be of length 1.
\begin{enumerate}
\item For each $w_i$ we construct a zigzag snake graph $\calg_i$ with $\ell (w_i)+1$ tiles where $\ell(w_i)$ is the number of direct or inverse arrows in $w_i$. Let $\calg_i$ be the zigzag snake graph with  tiles $T^i_1, \ldots, T^i_{\ell (w_i)+1}$ such that $T^i_2$ is glued to the right (resp. on top) of $T^i_1$ if $w_i$ is direct  (resp. inverse). 
\item We now glue  $\calg_{i+1}$  to $\calg_i$, for all $i$,  by identifying the last tile $T^i_{\ell (w_i)+1}$ of $\calg_i$ and the first tile $T^{i+1}_1$ of $\calg_{i+1}$ such that $T^i_{\ell (w_i)},  T^i_{\ell (w_i)+1}, T^{i+1}_2$ is a straight piece.   
\end{enumerate}

 We call this abstract snake graph  \textit{the snake graph  $\calg (w)$ associated} to $w$.

\begin{example}
The snake graph $\calg(w)$ associated to the abstract string $ w=\to \to \to \leftarrow$ is given by $
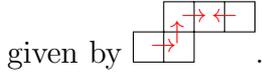
\begin{tikzpicture}[scale=.4]
\draw (0,0)--(2,0)--(2,2)--(4,2)--(4,1)--(0,1)--(0,0) (1,0)--(1,2)--(2,2) (3,1)--(3,2);
\node[color=red,rotate=0,scale=.8] at (1,.5){$\rightarrow$};
\node[color=red,rotate=90,scale=.8] at (1.5,1){$\rightarrow$};
\node[color=red,rotate=0,scale=.8] at (2,1.5){$\rightarrow$};
\node[color=red,rotate=0,scale=.8] at (3,1.5){$\leftarrow$};
\end{tikzpicture}$.
\end{example}

\begin{remark} Let $\iota :  \{ \to, \leftarrow, \emptyset \} \longrightarrow \{ \to, \leftarrow, \emptyset \}$ be the function defined by $\iota(\to) = \leftarrow$, $\iota (\leftarrow) = \to$ and $\iota(\emptyset) = \emptyset$. 
Given a sequence $w = a_1 \ldots a_n$ of arrows and inverse arrows, the inverse sequence is given by $\iota (w) = \iota(a_n) \ldots \iota(a_1)$.  So if $w$ is an abstract string, then $\calg(\iota(w))$ is obtained from $\calg(w) $ by reflecting it along $y=-x$.
\end{remark}

\subsection{String modules and snake graphs}

Let $K$ be a field, $Q$ a finite quiver, $I $ an admissible ideal in $KQ$ and let  $A=KQ/I$. We fix this notation throughout the paper. Denote by $Q_0$ the set of vertices of $Q$ with the convention that $Q_0 \subset \mathbb{N}$ and by $Q_1$ the set of arrows of $Q$. 
For $a \in Q_1$, let $s(a)$ be the start of $a$ and $t(a)$ be the end of $a$.

For each arrow $a$ in $Q_1$ we define the formal inverse $a^{-1}$ such that $s(a^{-1}) = t(a)$ and $t(a^{-1}) = s(a)$. A word $w  =
a_1 a_2 \ldots a_n$ is a \emph{string} if either $a_i$ or
$a_i^{-1}$ is an arrow in $Q_1$, if $s(a_{i+1}) = t(a_{i}),$  if
$a_{i+1} \neq a_{i}^{-1}$ for all $ 1 \leq i \leq n-1$ and if no subword of $w$ or its inverse is in $I$. Let $s(w) = 
s(a_1)$ and $t(w) = t(a_n)$.
Denote by $\mathcal{S_A}$ the set of strings modulo the equivalence relation $w \sim w^{-1}$, where $w$ is a string.

Given a string $w$ in $\cals_A$, we denote by $M(w)$ the corresponding string module over $A$.  Recall from \cite{BR, WW} that by definition a string module $M$ is given by an 
orientation of a type $\mathbb{A}$ Dynkin diagram where every vertex is replaced by a copy of $k$ and the arrows correspond to the identity maps. Note that $ M(w) \simeq  M(w^{-1})$.  The string module corresponding 
to a trivial string given by a vertex $i$ in $Q$ is the simple $A$-module corresponding to $i$. We call \emph{canonical embedding} of a submodule $N$ of $M$, the  injective map $N \to M$ induced by the identity on the non-zero components of $N$. For brevity, when we speak of a  submodule of a string module, we always mean a canonical submodule.

Given a string $w= a_1 \ldots a_n $ in $\cals_A$, let $v_1, \ldots, v_{n+1}$ be the vertices in $Q_0$,  that is $w = v_1 a_1 v_2 \ldots v_n a_n v_{n+1}$. We can view $w$ as an abstract string by forgetting the vertices and let $\calg$ be the abstract snake graph associated to $w$ as an abstract string.  After renumbering the tiles, this process gives a snake graph  with $n+1$ tiles $T_1, \ldots, T_{n+1}$ and we associate the face weight  $v_i$ to the tile $T_i$. We again call this snake graph with face weights  \textit{the snake graph  $\calg (w)$ associated} to $w$.

\begin{example}\label{Ex:Main}  We construct the snake graph associated to the string $w=1\rightarrow 2\rightarrow 3\rightarrow 4 \leftarrow 5\rightarrow 6$ with maximal alternating substrings $w_1=\rightarrow\rightarrow\rightarrow, w_2=\leftarrow$ and $w_3=\rightarrow$ below.

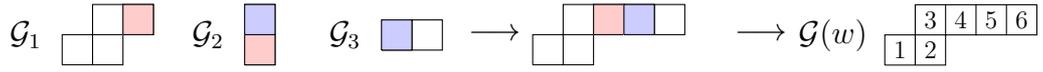
\[
\begin{tikzpicture}
\node at (1,0){$\begin{tikzpicture}[scale=.4]
\node at (-1.2,1){$\calg_1$};
\draw (0,0)--(2,0)--(2,2)--(3,2)--(3,1)--(0,1)--(0,0) (1,0)--(1,2)--(2,2);
\draw[fill=red!20!white] (2,1)--(2,2)--(3,2)--(3,1)--(2,1);
\end{tikzpicture}$};
\node at (3,0){$\begin{tikzpicture}[scale=.4]
\node at (-1.2,1){$\calg_2$};
\draw (0,0)--(1,0)--(1,2)--(0,2)--(0,0) (0,1)--(1,1);
\draw[fill=red!20!white] (0,0)--(1,0)--(1,1)--(0,1)--(0,0);
\draw[fill=blue!20!white] (0,1)--(1,1)--(1,2)--(0,2)--(0,1);
\end{tikzpicture}$};
\node at (5,0){$\begin{tikzpicture}[scale=.4]
\node at (-1.2,0.5){$\calg_3$};
\draw (0,0)--(2,0)--(2,1)--(0,1)--(0,0) (1,0)--(1,1);
\draw[fill=blue!20!white] (0,0)--(1,0)--(1,1)--(0,1)--(0,0);
\end{tikzpicture}
$};
\node at (6.5,0){$\longrightarrow$};
\node at (8,0){$\begin{tikzpicture}[scale=.4]
\draw (0,0)--(2,0)--(2,2)--(3,2)--(3,1)--(0,1)--(0,0) (1,0)--(1,2)--(2,2)
(3,1)--(5,1)--(5,2)--(3,2) 
(4,1)--(4,2);
\draw[fill=red!20!white] (2,1)--(2,2)--(3,2)--(3,1)--(2,1);
\draw[fill=blue!20!white] (3,1)--(3,2)--(4,2)--(4,1)--(3,1);
\end{tikzpicture}$};
\node at (10,0){$\longrightarrow$};
\node at (12,0){$\begin{tikzpicture}[scale=.4]
\node at (-1.7,1){$\calg(w)$};
\draw (0,0)--(2,0)--(2,2)--(3,2)--(3,1)--(0,1)--(0,0) (1,0)--(1,2)--(2,2)
(3,1)--(5,1)--(5,2)--(3,2) 
(4,1)--(4,2);
\node[scale=.8] at (.5,.5){$1$};
\node[scale=.8] at (1.5,.5){$2$};
\node[scale=.8] at (1.5,1.5){$3$};
\node[scale=.8] at (2.5,1.5){$4$};
\node[scale=.8] at (3.5,1.5){$5$};
\node[scale=.8] at (4.5,1.5){$6$};
\end{tikzpicture}$};
\end{tikzpicture}
\]
\end{example}

\begin{remark}\label{tops and socles} Let $\calg(w)$ with tiles $T_1, \ldots, T_{n+1}$ be a snake graph associated to a string module $M(w)$ given by the zigzags $\calg_1, \ldots, \calg_k$.
For all $i$,  $\calg_i \cap \calg_{i+1}$ is a tile corresponding to a simple module.  Furthermore, each such simple module is either in the top or socle of $M(w)$ and if we add the simples corresponding to $T_1$ and $T_{n+1}$ then any simple in the top or socle of $M(w)$ is of that form.
\end{remark}

\subsection{Construction of the string of a snake graph}\label{subsec:string module of snake graph}
 
Given a snake graph $\calg$ with consecutive tiles $T_1, \ldots, T_{n+1}$, we define a function 
\[
f_+ : \{ (T_1, T_2), (T_2, T_3), \ldots, (T_{n}, T_{n+1}) \} \longrightarrow \{\to,\leftarrow \}
\] 
recursively as follows: Set  $f_+(T_1, T_2) =\to$. Now suppose that $(T_{j-1}, T_j) \mapsto \varepsilon$, for  $\varepsilon \in \{\to,\leftarrow\}$,  then define   
\[
(T_j, T_{j+1})  \mapsto \left\{ \begin{array}{ll}
\varepsilon & \mbox{ if $T_j, T_{j+1}, T_{j+2}$ is a zigzag } \\
 \varepsilon^c  &  \mbox{ if $T_j, T_{j+1}, T_{j+2}$ is a straight piece }
\end{array} \right.
\]
 where $\varepsilon^c$ is the complement of $\varepsilon$ in $\{\to,\leftarrow\}$.

We define the function $f_-: \{ (T_1, T_2), (T_2, T_3), \ldots, (T_{n}, T_{n+1}) \} \longrightarrow \{\to,\leftarrow \}$ by setting $f_-(T_1, T_2) = \leftarrow$ and proceeding as above. 

We call $f_+$ and $f_-$ \textit{arrow functions} of $\calg$. Note that these functions are referred to as sign functions in \cite{CS1} and they recover the abstract string of a snake graph.

\begin{example}\label{Ex:ArrowFnc} The arrow function $f_+$ of the snake graph $\calg=\calg(w)$ in \emph{Example}~\ref{Ex:Main} is given by $(\to,\to,\to,\leftarrow,\to)$ and the arrow function $f_-$ is given by $(\leftarrow,\leftarrow,\leftarrow,\to,\leftarrow).$ 
\end{example}

Suppose that $\calg$ has face weights $v_1, \ldots, v_{n+1}$. Define a string $w^+= a_1  a_2 \ldots  a_n $ by defining  $a_i$ to be an arrow   with $s(a_i) = v_i$ and $t(a_i) = v_{i+1}$ if $f_+(T_i, T_{i+1}) =\to$ or   if $f_+(T_i, T_{i+1}) =\leftarrow$
 by defining $a_i$ to be the formal inverse of an arrow with  $s(a_i) = v_{i+1}$ and $t(a_i) = v_{i}$. 
 
Define the string $w^-$ in an analogous way using the arrow function $f_-$. 

We call $w^+$ and $w^-$ the \textit{abstract strings  associated} to $\calg$.  
 
The lemma below immediately follows from the definitions above. 
 
 \begin{lemma}\label{lem:SG-string}
 Let $w = a_1 \ldots a_n$ be a string in $\cals_A$ and let $\calg=\calg (w)$ be the associated snake graph. If $a_1 \in Q_1$ then the string $w^+$ associated to $\calg$ is equal to $w$ and if $a_1^{-1} \in Q_1$  then the string $w^-$ associated to $\calg$ is equal to $w$.  
 \end{lemma}

 In view of Lemma~\ref{lem:SG-string} we freely interchange strings and snake graphs.

\begin{example} The string $w^+$ associated to the snake graph $\calg$ in \emph{Example}~\ref{Ex:Main} is given by $1 \rightarrow 2 \rightarrow 3\rightarrow 4\leftarrow 5\rightarrow 6$ and agrees with $w.$
\end{example}

\subsection{Maximal and minimal perfect matchings and symmetric differences}\label{min max sym diff}  Let $w$ be an abstract string, define the \textit{minimal perfect matching} $P_{min}$ of the snake graph $\calg (w)$ to be the matching  consisting only of boundary edges and such that if locally we have the following sequence $\rightarrow \leftarrow$ in $w$ then  the corresponding tile in $\calg(w)$ has two boundary edges in the matching.   If $w=\to$ ($w=\leftarrow$) is given by a single letter then the \emph{minimal perfect matching} $P_{min}$ of the corresponding two tile snake graph $\calg(w)$ contains two boundary edges on the first (second) tile. The maximal perfect matching $P_{max}$ of $ \calg$ is given by the perfect matching consisting of all boundary edges not in $P_{min}$. Clearly $P_{max}  \cup P_{min}$ corresponds to the set of  all  boundary edges of $\calg$.

\begin{remark}
In terms of string modules, if $w \in \cals_A$ and if a vertex $v$ in $w$ corresponds to a simple  in the socle of $M(w)$  then the corresponding tile in the minimal perfect matching of $\calg$ has two boundary edges in the matching.  
\end{remark}

We define the symmetric difference  $P \ominus P'$ of any two matchings $P$ and $P'$ of $\calg$  to be the set of edges of $\calg$  given by $(P \cup P') \setminus ( P \cap P')$.

It follows from \cite{MSW1} that $P \ominus P_{min}$ gives rise to a set of enclosed tiles of $\calg$ and  we denote by $\cup \calh_i$ the union of the snake subgraphs of $\calg$ corresponding to the tiles enclosed by $P \ominus P_{min}$. This gives a canonical embedding $\varphi$ of $P \ominus  P_{min} = \cup \calh_i$ into $\calg$.

\begin{remark}\label{remark correspondence} 
Given a perfect matching of a snake graph $\calg$ corresponding to an arc in a triangulation of an oriented marked surface and given  the corresponding minimal or maximal perfect matching, in \cite{MSW1} the symmetric difference is used to determine the $y$-coefficients of the cluster variable associated to the arc. This  thereby also  gives dimension vectors of submodules of the string modules associated to the arc and the snake graph $\calg$. 
\end{remark}

\section{Perfect matching lattices, canonical submodule lattices and intervals in the weak Bruhat order}

Let $\calg$ be a snake graph with face weights $\{v_1, \ldots, v_n \} \subset \mathbb{N} $. Define a function  $h : \calg \to \mathbb{N}^k$ given by $h(\calg) =  (n_1, \ldots, n_k) $ where  $n_i$ equals the number of tiles of face weight $i$ in $\calg$. We call $h$ the \textit{face   function} of $\calg$. 

\subsection{Perfect matchings of a snake graph and submodules}\label{sec::pm of a snake graph and submodules}

In the context of modules and submodules, the correspondence in Remark~\ref{remark correspondence}  can be made more explicit.  Namely,

\begin{proposition}\label{Matching to submodule}
Let $A=KQ/I$ with $I$ admissible in $KQ$. Let $w$ be a string, $M(w)$ the corresponding string  module over $A$ and let $P$ be a perfect matching of the snake graph $\calg (w)$.    Then $P \ominus P_{min}$ gives rise to the canonical embedding  $\varphi: M(P) \hookrightarrow M(w)$   such that $h(P \ominus  P_{min})$ is the dimension vector of the  submodule $M(P)$ of $M(w)$.
\end{proposition} 

\begin{proof} 
We have that $P \ominus P_{min}  = \cup \calh_j $ where each $\varphi_i : \calh_j \hookrightarrow \calg(w)$  is  a  canonical embedding of  snake graphs.  Let $T_1, \ldots,  T_{n_j}$ be the tiles of $\calh_j$. Furthermore, the arrow function $f$ of $\calg (w)$   is determined by whether the first letter in $w$ is a direct or inverse arrow. This induces an arrow function $f_j$  for each $\calh_j$ such that $f_j = g_+$ if $f\vert_{\calh_j} (T_1, T_2) = \to$ and $f_j = g_-$ if $f\vert_{\calh_j} (T_1, T_2) = \leftarrow$ where $g_{\pm}$ is the arrow function of $\calh_j$. Let $w_j$ be the string associated to $\calh_j$ determined by $f_j$. 

 We will show that each $M (w_j)$ is a submodule of $M(w)$. 

We prove that each $\calh_j$ contains at least one socle of $M(w)$.  If $n_j=1$ then $\calh_j$ consists of a single tile $T_1$, so $P_{min} \vert_{T_1} \cup P \vert_{T_1}$ matches all boundary edges of $T_1$ and by definition, the corresponding simple is in the socle of $M(w)$. Now let $n_j \geq 2$ and suppose that $f_j = g_+$, that is $f_j(T_1, T_2) = \to$. Then the simple corresponding to $T_1$ is not a socle in $M(w)$. Therefore $P_{min} \vert_{T_1}$ must be a single edge of $T_1$. Now if $ n_j=2$ or if $T_1, T_2, T_3$ is a straight piece then since $P_{min}$ is a perfect matching with only boundary edges,  we have that $P_{min} \vert_{T_2}$ matches two boundary edges of $T_2$ and by definition $T_2$ corresponds to a socle of $M(w)$. 
So suppose that $T_1, \ldots, T_k$ is a maximal zigzag piece then either $ k = n_j$ or $T_{k-1}, 
T_k, T_{k+1}$ is a straight piece. Furthermore, without loss of generality we can assume that $T_2$ is to the right of $T_1$ and $P_{min} \vert_{T_2}$ is at the bottom of $T_2$. That is  $P_{min}$  matches the bottom edges of $T_2, \ldots, T_{2i}$ and the left edges of $T_1, \ldots, T_{2i-1}$.  If $k$ is even then the top and bottom of $T_k$ are boundary edges and  $P_{min} \vert_{T_k}$ matches the top and bottom boundary edges of  $T_k$ and if  $k$ is odd then the left and right of $T_k$ are boundary edges and  $P_{min} \vert_{T_k}$ matches the left and right boundary edges of $T_k$ and therefore $T_k$ corresponds to a simple in the socle of $M(w)$.

  Now suppose that $f_j = g_-$, that is $f_j(T_1, T_2) = \leftarrow$. Now suppose that $ n_j =1$. Then clearly as above the  corresponding simple lies in the socle of $M(w)$. 
So suppose that $T_1, \ldots, T_k$ is a maximal zigzag piece. Furthermore, without loss of generality we can assume again that $T_2$ is to the right of $T_1$.  Then $f_j(T_{k-1}, T_k) = \leftarrow$. Then this implies by our convention that $P_{min}$ does not match two boundary edges of $T_k$. Then if $k$ is even, $T_{k}$ is to the right of $T_{k_1}$ and the top edge of $T_{k_1}$  is a boundary edge which is matched in $P_{min}$ and the top edges of $T_{k-3}, \ldots, T_1$ are  boundary edges which are matched in $\Pm$. Since $T_1$ is to the left of $T_2$, the bottom edge of $T_1$ is also a boundary edge that is matched in $\Pm$. Hence it corresponds to a simple in the socle of $M(w)$. A similar argument shows that the top and bottom edges of $T_1 $ are matched in $\Pm$ if $k$ is odd.

By symmetry an analogous argument starting with $ T_{n_j}$ implies that $N(w_j)$ is a submodule of $M(w)$. Furthermore, each canonical embedding $\varphi_i$ determines a tile $T_{j_i}$ of $\calg$ which corresponds to the first tile of $\calh_i$ as a subgraph of $\calg$. The tile $T_{j_i}$ corresponds simultaneously to $s(w_i)$ and some vertex in $w$, identifying these two vertices gives rise to an embedding $N(w_i) \hookrightarrow M(w)$ for each $i$, in turn defining  the canonical embedding $\varphi: N \hookrightarrow M(w)$. 
\end{proof}

\begin{remark}\label{rem:MaxMatch} In Proposition~\ref{Matching to submodule} if $P = P_{max}$ then $M(P) = M(w)$ and we refer this module also as the string module $M(\calg)$ associated to the snake graph $\calg$. 
\end{remark}

\begin{definition} 
We call the submodule $M(P)$ in Proposition~\ref{Matching to submodule} the {\it canonical submodule associated to $P$.} 
\end{definition} 

Conversely,  given a string module and a canonical submodule, we want to associate a perfect matching. To start with let $N = N(w_1) \oplus \cdots \oplus  N(w_n)$ be a canonical submodule of a string module $M(w)$  with canonical embedding $\varphi$. Let $\calh_i$ be a snake subgraph of $\calg = \calg (w)$ given by $w_i$  corresponding to the embedding $\varphi$. Note that $\calh_i$ is not necessarily unique, if $w$ contains several copies of $w_i$. Denote the maximal perfect matching of $\calh_i$ by $P_{max}(\calh_i)$.   Set  
\[
P_{\varphi} = (\bigcup_{i=1}^n P_{max} (\calh_i)) \cup \Pm \vert_{\calg \setminus (\bigcup_{i=1}^n \calh_i)}. 
\]

\begin{proposition}\label{Submodule to matching}
Keeping the notation above, the set of edges $P_{\varphi}$ of $\calg$ is a perfect matching of $\calg$ and $N = M(P_\varphi)$.
\end{proposition}

\begin{proof}

If we prove that  $P_{min} \vert_{\calh_i}$ is a perfect matching of $\calh_i$ then 
all vertices of $\calh_i$ are matched by matching edges belonging to $\calh_i$. Therefore we can replace  $P_{min} \vert_{\calh_i}$ inside $P$  by any other matching of $\calh_i$ 
and the result is still a perfect matching of $\calg$. In particular this is true if we replace $P_{min} \vert_{\calh_i}$ by $P_{max} (\calh_i)$.

We now prove that $P_{min} \vert_{\calh_i}$ is the minimal perfect matching of  $\calh_i$. Let $T_1, \ldots, T_n$ be the tiles of $\calg$ and let $T_j, \ldots, T_{j+l+1}$ be the tiles of $\calh_i$ considered as a subgraph of $\calg$. Without loss of generality assume that $T_{j+1}$ is to the right of $T_j$.
Let $f_i$ be the sign function on $\calh_i$ induced by $w_i$. 

If $f_i$ is such that $f_i(T_j, T_{j+1}) = \to$ then this implies that $T_j$ does not correspond to a simple in the socle of $M(w)$ (or $N(w_i)$).  So there must be a maximal zigzag piece $T_j, \ldots, T_k$ in $\calh_i$ such that $T_{k}$ corresponds to a simple in the socle of $M(w)$. If $T_{k}$ is to the right of $T_{k-1}$ then the top and bottom edges of $T_k$ are matched in $\Pm(\calg)$ and in $\Pm(\calh_i)$ and $k$ is even. This implies that the bottom edge of $T_{j+1}$ is also matched both in    $\Pm(\calg)$ and in $\Pm(\calh_i)$. Furthermore, the left edge of $T_j$ in $\Pm(\calh_i)$ is matched. If  it is not matched in $\Pm(\calg)$ then $T_{j-1}$ is to the left of $T_j$ and $f(T_{j-1}, T_j) = \leftarrow$ where $f$ is the arrow function of $\calg$. This implies that $T_j$ corresponds to a simple in the top of $\calg$ and this is a contradiction since $T_j$ corresponds to the start of the string of a submodule and $j \neq 1$. 

If $f_i$ is such that $f_i(T_j, T_{j+1}) = \leftarrow$ then this implies that $T_j$  corresponds to a simple in the socle of  $N(w_i)$. But any simple in the socle in $N(w_i)$ is in the socle of $M(w)$.  Therefore $T_j$ has two boundary edges matched and  $P_{min} \vert_{\calh_i}$ is the minimal perfect matching of  $\calh_i$. 

To prove the last statement, observe that 
by the definition of $P_\varphi$, we have that $\Pm \ominus P_\varphi = \bigcup_{i=1}^n P_{max} (\calh_i)$. By remark~\ref{rem:MaxMatch} $M( P_{max} (\calh_i)) = N(w_i)$ for all $i$.
\end{proof}

\begin{example}
Consider the snake graph 
$\calg_w=\begin{tikzpicture}[scale=.4]
\draw (0,0)--(2,0)--(2,2)--(4,2)--(4,1)--(0,1)--(0,0) (1,0)--(1,2)--(2,2) (3,1)--(3,2);
\node[scale=.8] at (.5,.5){$1$};
\node[scale=.8] at (1.5,.5){$2$};
\node[scale=.8] at (1.5,1.5){$3$};
\node[scale=.8] at (2.5,1.5){$4$};
\node[scale=.8] at (3.5,1.5){$5$};
\end{tikzpicture}$ associated to the string $w=1\rightarrow 2 \rightarrow 3 \rightarrow 4 \leftarrow 5$. The submodules of $M_w$ corresponding to the perfect matchings 
$\begin{tikzpicture}[scale=.4]
\draw (0,0)--(2,0)--(2,2)--(4,2)--(4,1)--(0,1)--(0,0) (1,0)--(1,2)--(2,2) (3,1)--(3,2);
\draw[color=red,ultra thick] (0,0)--(0,1) (1,0)--(2,0) (1,1)--(2,1) (1,2)--(2,2) (3,1)--(3,2) (4,1)--(4,2);
\node[scale=.8] at (.5,.5){$1$};
\node[scale=.8] at (1.5,.5){$2$};
\node[scale=.8] at (1.5,1.5){$3$};
\node[scale=.8] at (2.5,1.5){$4$};
\node[scale=.8] at (3.5,1.5){$5$};
\end{tikzpicture}$
and
$\begin{tikzpicture}[scale=.4]
\draw (0,0)--(2,0)--(2,2)--(4,2)--(4,1)--(0,1)--(0,0) (1,0)--(1,2)--(2,2) (3,1)--(3,2);
\draw[color=red,ultra thick] (0,0)--(0,1) (1,0)--(2,0) (1,1)--(1,2) (2,1)--(2,2) (3,1)--(3,2) (4,1)--(4,2);
\node[scale=.8] at (.5,.5){$1$};
\node[scale=.8] at (1.5,.5){$2$};
\node[scale=.8] at (1.5,1.5){$3$};
\node[scale=.8] at (2.5,1.5){$4$};
\node[scale=.8] at (3.5,1.5){$5$};
\end{tikzpicture}$
are $3 \rightarrow 4$ and the simple module $4$, respectively. 

\end{example}

\begin{definition}\label{def::canonical submodule lattice} Let $A$ be a finite dimensional $K$-algebra and let $M$ be an $A$-module. 
Define the {\it canonical submodule lattice ${\mathcal L}(M) $ of $M$} to be the partially ordered set given by all canonically embedded submodules with embedding $\varphi$ denoted by $(N, \varphi)$ of $M$ with partial order given by  $(N, \varphi) \leq (N', \varphi')$ if $\varphi' \vert_N = \varphi$. 
\end{definition} 

 \begin{remark}
 Given  $(N, \varphi)$ and $(N', \varphi')$ in $\call (M)$
 \begin{itemize}
 \item the join $(N, \varphi) \vee (N', \varphi')$ is the smallest $(L, \theta)$ such that  $N, N'$ are submodules of $L$ and $\theta \vert_N  = \varphi$ and $\theta \vert_{N'} =\varphi'$, 
 \item the meet  $(N, \varphi) \wedge (N', \varphi')$ is the largest $(L, \theta)$ such that $L$ is a submodule of $N$ and $N'$ and $\varphi \vert_L = \theta$  and $\varphi' \vert_L = \theta. $
 \end{itemize}
 \end{remark}

 \begin{remark} \label{rm::lattice squares} 
Let $H(M)$ be the Hasse diagram of $\call (M)$. 
 \begin{enumerate} 
 \item  Let  $(N, \varphi) \leq (N', \varphi')$ be such that they give an edge in $H(M)$. Then this edge is labelled by the module $N' / N$ which is simple.
 \item  If in $H (M)$ there are edges $i, j$ such that 
$\scalebox{.5}
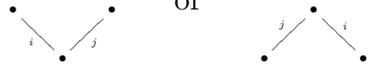
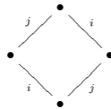
{\xymatrix{  & \bullet && \bullet &  \\ 
&&\ar@{-}[ul]^{i}  \bullet \ar@{-}[ur]_j && \\  }}$  or 
$\scalebox{.5}{\xymatrix{  && \bullet && \\
& \bullet  \ar@{-}[ur]^j && \bullet \ar@{-}[ul]_i  &  }}$ 
then there is a mesh 
$\scalebox{.5}{\xymatrix{  && \bullet && \\
& \bullet \ar@{-}[ur]^j && \bullet \ar@{-}[ul]_i  &  \\ 
&&\ar@{-}[ul]^i \bullet \ar@{-}[ur]_j && \\ } }$ in $H (M)$.
\end{enumerate}
\end{remark}

In the following  we will not distinguish between a lattice and the associated Hasse diagram.

\subsection{Bijection between perfect matching lattice and canonical submodule lattice}

The next result,  Theorem~\ref{lattice correspondence}, 
gives an isomorphism of the perfect matching lattice of a snake graph and the canonical submodule lattice of the corresponding string module. 
This has also been noted in \cite[Remark 5.5]{MSW2}.  
Using  Proposition~\ref{Matching to submodule} and Proposition~\ref{Submodule to matching}, the 
 proof of  Theorem~\ref{lattice correspondence} is based  on the idea that the edge labelling of the Hasse diagram of the perfect matching lattice corresponds to    adding or removing a simple top of the  corresponding submodules in the extended submodule lattice.

\begin{theorem}\label{lattice correspondence}
Let $A = KQ/I$  and let  $M(w)$ be a string module over $A$ with string $w$ and  with associated snake graph $\calg $. Then the perfect matching lattice ${\mathcal L}(\calg) $  of $\calg$   is in bijection with the canonical submodule  lattice ${\mathcal L}(M) $.
\end{theorem}

\begin{proof}
By Proposition~\ref{Matching to submodule} and Proposition~\ref{Submodule to matching} the two lattices  are equal as sets. Let $H(\calg)$ and $H(M)$ be the Hasse diagrams of   ${\mathcal L}(\calg) $ and ${\mathcal L}(M) $, respectively.   We now show that there is an edge between two vertices  in  $H(\calg) $ if and only if there is an edge between the corresponding two vertices in  $H(M) $. 

Suppose that $P$ and $P'$ are two perfect matchings  of $\calg$ that are connected by an edge
 in $H(\calg)$  and let $w_P$ and $w_{P'}$ be the strings of $M(P)$ and $M(P')$ 
 respectively. Let $P_{min} \ominus P = \cup_{i} \calh_i$  and 
 $P_{min} \ominus P' = \cup_{i} \calh'_i$. Then $P$ and $P'$ determine canonical embeddings $\varphi :  \cup_{i} \calh_i \hookrightarrow \calg$ and $\varphi' :  \cup_{i} \calh'_i \hookrightarrow \calg$, respectively. The perfect matchings $P$ and $P'$ agree everywhere except on one tile $T$ of $\calg$ 
 where they have opposite matchings. This implies that at least one edge in exactly one 
 of them agrees with $P_{min}$ and the other one has at least one edge agreeing with 
 $P_{max}$. Therefore the vertex $v$  in $w$ corresponding to $T$ is either in $w_P$ 
 or $w_{P'}$ and $w_P$ and $w_{P'}$ agree everywhere except in the arrows or 
 inverses of arrows connected to $v$.  
 Without loss of generality suppose $v$ is in $w_P$. Suppose that the simple 
 corresponding to   $v$ is not in ${\rm top}(M(w_P))$.  Removing $v$ from $w_P$ gives 
  rise to $w_{P'}$  giving rise to at least one substring $m$  of $w_{P'}$ such that 
  $M(m)$ is not a submodule of $M$, leading to a contradiction. Therefore $v$ 
  corresponds to a simple in the top of $M(w_P)$ and removing $v$ corresponds to an 
  edge in $H(M)$ between $M(P)$ and $M(P')$. 
 Furthermore,   under our assumption, $\cup_{i} \calh_i \setminus \cup_{i} \calh'_i = T$, therefore $\varphi \vert_{  \cup_{i} \calh'_i} = \varphi'$. This implies that the induced module embeddings are such that $\varphi \vert_{M(P')} = \varphi'$.
  
Conversely, suppose that $ (N(w_i), \varphi_i)$ and $(N(w_j), \varphi_j)$  are two canonically embedded submodules of $M(w)$ connected by an edge in $H(M)$. Therefore, without loss of generality,  by definition of $H(M)$, there exists exactly one vertex $v$ contained in $w_i$ and not contained in $w_j$ and all other vertices of $w_i$ and $w_j$ are the same. Let $P_{\varphi_i}$ and $P_{\varphi_j}$ be the perfect matchings associated to $ (N(w_i), \varphi_i)$ and $(N(w_j), \varphi_j)$, respectively. Let $(P_{min} \ominus P_{\varphi_i})  \setminus (P_{min} \ominus P_{\varphi_j}) = T$ for some tile $T$ of $\calg$.  
Therefore $P_{min}\vert_T = P_{\varphi_j} \vert_T$ and $P_{\varphi_i} $ and $ P_{\varphi_j} $ agree on all tiles different from $T$. Therefore  by \cite{MSW2} there is an edge between $P_{\varphi_i} $ and $P_{\varphi_j} $ in $H(\calg)$.
\end{proof}

We now recall \cite[Theorem 13.1]{MSW1} in the context of our set-up using the  correspondence in Theorem~\ref{lattice correspondence}.

\begin{corollary} 
$A = KQ/I$ with $k$ simple modules, $M$ a  string $A$-module with snake graph $\calg$ and let $\underline{e}  \in   \mathbb N^k$.  Then the number of canonical submodules of $M$ with dimension vector $\underline{e}$ is given by 
\[ 
n(M, \underline{e}) =  \bigm|  \{ P \in \match (\calg) \mid h(P \ominus P^-) = \underline{e} \}  \bigm|.
\]
Therefore the Euler characteristic of the quiver Grassmanian  of $M$ with dimension vector $\underline{e}$ is given by 
\[
\chi (Gr_e (M) =  n(M, \underline{e}).
\]
\end{corollary}

\subsection{Join irreducibles in the canonical submodules lattice}\label{Sec:JoinIrr}
By \cite[Theorem 5.2]{MSW2}, the perfect matching lattice of a snake graph is distributive. Therefore by Theorem~\ref{lattice correspondence} the canonical submodule of a string module is distributive and it is graded by the dimensions of the submodules.  

We briefly recall some facts about distributive lattices. 

\begin{definition}  
A lattice $\call$ is \emph{distributive} if for all $x,y,z \in \call$ we have $x \vee (y \wedge z) = (x \wedge y) \vee (x \wedge z)$ and  $x \wedge (y \vee z) = (x \vee y ) \wedge (x \vee z)$. 
\end{definition}

By Birkoff's Theorem for any finite distributive lattice $\call$ there exists a unique poset $P$ (up to isomorphism) such that $\call$ is given by the lattice of order ideals of $P$. Furthermore, $P$ is isomorphic to the poset induced by the set of 
join irreducible elements of $\call$. 

We recall that $I$ is an \emph{order ideal} of  $P$ if for any $x \in P$, if $y \leq x$ then $x \in  I$. 
An element $a$ in a lattice $\call$ is \emph{join irreducible} if $a = x \vee y$ implies $a =x  $ or $a =y$, for $x,y \in \call$. Furthermore, if $\call $ is a finite lattice then $a \in \call$ is join irreducible if and only if $a$ has exactly one lower cover. 

The following is a direct consequence of the definitions. 

\begin{proposition}
Let $M$ be an abstract string module with canonical submodule lattice $\call(M)$. The join irreducible elements in $\call(M)$ are the canonical submodules $N$ of $M$ such that ${\rm top}(M)$ is simple. 
\end{proposition}

We now show how the poset structure of  the set of join irreducibles can be read-off directly from the  string. More precisely, let $M$ be an abstract string module with string $w = w_1 \ldots w_n$ where the $w_i$ are alternating direct and inverse strings. Suppose that $w_1 = a_1 \ldots a_j$ is a direct string with $a_i \in Q_1$, for $1 \leq i \leq j$. Then $M(t(w_1))$ and $M(a_i \ldots a_j)$, for $ 1 \leq i \leq j$  are join irreducible submodules of $M$ and $$M(t(w_1)) \subset M(a_j) \subset \cdots \subset M(a_2 \ldots a_j) \subset M(a_1 \ldots a_j) = M(w_1).$$ 
Now let $w_2 = b_1 \ldots b_h$ and $w_3 = c_1 \ldots c_k$ with $b_i^{-1}, c_i \in Q_1$. Then $M(w_2 w_3), M(b_1 \ldots b_i)$, for $1 \leq i \leq h-1$ and $M(c_i \ldots c_k)$, for $2 \leq i \leq k$ as well as $M(t(w_3))$ are join irreducible. Furthermore, $$M(t(w_1)) = M(s(w_2)) \subset M(b_1) \subset \cdots \subset M(b_1 \ldots b_{h-1}) \subset M(w_2 w_3)$$ and  $$M(t(w_3)) \subset M(c_k) \subset \cdots \subset M(c_3 \ldots c_k) \subset M(c_2 \ldots c_{k}) \subset  M(w_2 w_3).$$ 
And similarly for any substring $w_r w_{r+1}$ of $w$  with $r$ even. Now suppose that $w_n = d_1 \ldots d_m$ is an inverse string, that is $d_i^{-1} \in Q_1$. Then the $M(s(w_n)) = M(t(w_{n-1}), M(d_1 \ldots d_i)$, for $1 \leq i \leq m$ are join irreducible and 
$$M(s(w_n)) \subset M(d_1) \subset M(d_1 d_2)  \cdots \subset  M(d_1 \ldots d_m).$$ 
In case $w$ starts with an inverse (resp. ends with a direct) string, similar reasoning as above gives the following result.

\begin{corollary}\label{poset structure of join irreducibles}
Let $M$ be an abstract string module with string $w = w_1 \ldots w_n$ where the $w_i$ are alternating direct and inverse strings. Then the form of  $w$ gives the poset structure of the set of join irreducibles. 
\end{corollary}

\begin{remark}\label{rem::joinIrre}
(1) Instead of the direct proof above, Corollary~\ref{poset structure of join irreducibles} can also be deduced from Thereom~\ref{lattice correspondence} together with  \cite[Theorem 5.4]{MSW2}.

(2)  Another way of determining the poset structure of the join irreducibles is by relabelling the vertices of $w$ from $1$ to $t$ where $t$ is the number of vertices in $w$ and considering the string as an orientation of a quiver $Q$ of type $A_t$. Then the join irreducibles are given by relabelling  the vertices in the  projective indecomposable $KQ$-modules back to the initial label. We illustrate this in an example.
\end{remark}

\begin{example} Let $A$ be the path algebra of the quiver $ 
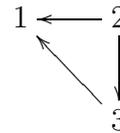
\xymatrix{
        1       & 2 \ar[l]   \ar[d] \\
        & \ar[ul]  3 }$. Consider the string module $M =  \begin{tikzpicture}[baseline=-0.85ex]
\node at (0,0){1};
\node at (.25,.25){2};
\node at (.5,0){3};
\node at (.75,-.25){1};
\node at (1,0){2};
\end{tikzpicture}$ over $A$. Following Remark~\ref{rem::joinIrre}~(2), relabel $M$ as $A_M=1\leftarrow 2 \rightarrow 3 \rightarrow 4 \leftarrow 5$. The poset of projectives in $A_5$ is given on the left of Figure~\ref{fig::joinIrre}. We make the identification $4\to 1$ and $5 \to 2$ to obtain the poset structure of join irreducibles of the canonical submodule lattice of $M$ given on the right of Figure~\ref{fig::joinIrre}.
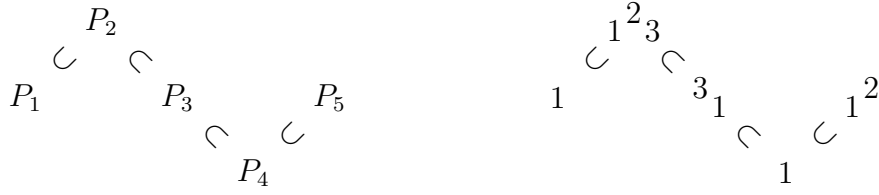
\begin{figure}[H]
\[
\begin{tikzpicture}[baseline=-0.85ex]
\node at (0,0){$P_1$};
\node[rotate=45] at (.5,.5){$\subset$};
\node at (1,1){$P_2$};
\node[rotate=-45] at (1.5,.5){$\subset$};
\node at (2,0){$P_3$};
\node[rotate=-45] at (2.5,-.5){$\subset$};
\node at (3,-1){$P_4$};
\node[rotate=45] at (3.5,-.5){$\subset$};
\node at (4,0){$P_5$};
\node at (7,0){$1$};
\node[rotate=45] at (7.5,.5){$\subset$};
\node at (8,1){$\begin{tikzpicture}[baseline=-0.85ex]
\node at (0,0){1};
\node at (.25,.25){2};
\node at (.5,0){3};
\end{tikzpicture}$};
\node[rotate=-45] at (8.5,.5){$\subset$};
\node at (9,0){$\begin{tikzpicture}[baseline=-0.85ex]
\node at (.5,0){3};
\node at (.75,-.25){1};
\end{tikzpicture}$};
\node[rotate=-45] at (9.5,-.5){$\subset$};
\node at (10,-1){$1$};
\node[rotate=45] at (10.5,-.5){$\subset$};
\node at (11,0){$\begin{tikzpicture}[baseline=-0.85ex]
\node at (1,0){2};
\node at (.75,-.25){1};
\end{tikzpicture}$};
\end{tikzpicture}
\]
\label{fig::joinIrre}
\end{figure}
\end{example}

\subsection{Lattice bijections with intervals in the weak Bruhat order} \label{sec::weak Bruhat}

We now fix some notation. Let $\frak{S}_{n+1}$ be the symmetric group on $n+1$ letters. It is generated by the simple reflections $s_i = (i \; i + 1) $  for $i = 1, \ldots, n$. Given an abstract string module $M$ corresponding to a snake graph $\calg$ with $n+1$ tiles, number the tiles linearly 1 through $n+1$ and adapt the corresponding notation for $M$.  Let $(i_1, \ldots, i_n)$ be the edge labellings in a  maximal chain in $\call(\calg)$ (resp. in $\call(M)$). Then $(i_1, \ldots, i_n)$ gives rise to the  element $s_{i_1} \cdots s_{i_n}$ in $\frak{S}_{n+1}$. Since $i_j \neq i_k$ for $j \neq k$,   $s_{i_1} \cdots s_{i_n}  $ is a reduced  expression of a Coxeter element  in $\frak{S}_{n+1}$. Note that if in $\call (M)$ we have a square $\scalebox{.5}
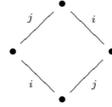
{\xymatrix{  && \bullet && \\
& \bullet \ar@{-}[ur]^j && \bullet \ar@{-}[ul]_i  &  \\ 
&&\ar@{-}[ul]^i \bullet \ar@{-}[ur]_j && \\ } }$, then $i, j$ correspond to simple modules in the top of  a submodule of $M$. Therefore, $|i - j| >1$ and this translates to the commutativity relation $s_i s_j = s_j s_i$ in $\frak{S}_{n+1}$. 

\begin{lemma}\label{lem::reduced expressions}
Let  $\sigma= s_{i_1} \cdots s_{i_n}$ and $\sigma'=s_{j_1} \cdots s_{j_n}$ be two reduced expressions in $\frak{S}_{n+1}$ corresponding to two distinct maximal chains in $\call (M)$ for some string module $M$. Then $\sigma = \sigma' $.
\end{lemma}

\begin{proof}
Let $k$ be the smallest integer such that $s_{i_k}  \neq s_{j_k}$. Then the edge labellings of the two maximal chains coincide in the first $k-1$ positions, that is $i_1 = j_1$, \ldots, $i_{k-1} = j_{k-1}$ and $i_k \neq j_k$.

Then in $\call (M)$ there is a corresponding sublattice which is of the form 

$$\scalebox{.8}
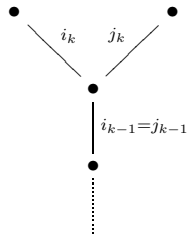
{\xymatrix{  & \bullet && \bullet&  \\ 
 &&\ar@{-}[ul]_{i_k}  \bullet \ar@{-}[ur]^{j_k} && \\
 && \bullet \ar@{-}[u]_{i_{k-1}= j_{k-1}} \ar@{.}[d]&&\\
 && \\
}}$$

Since $\call (M)$ is finite, we have  a sublattice of the form

$$\scalebox{.8}
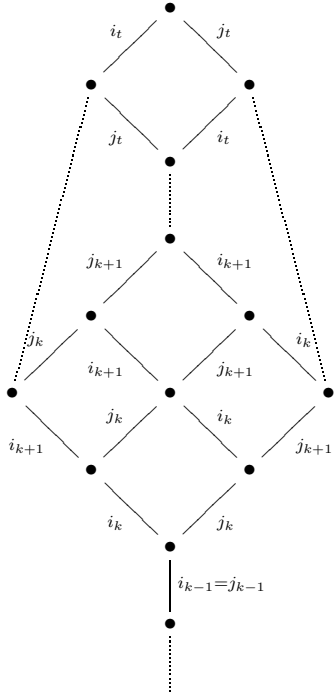
{\xymatrix{ &&\ar@{-}[dl]_{i_{t}} \bullet \ar@{-}[dr]^{j_t}&& \\
 &\bullet \ar@{-}[dr]_{j_{t}}  \ar@{.}[ddddl] &&\ar@{-}[dl]^{i_{t}} \bullet \ar@{.}[ddddr]&\\
 &&\bullet&&\\
&&\bullet\ar@{.}[u]&&\\ 
  &  \bullet  \ar@{-}[ur]^{j_{k+1}}  &  & \bullet  \ar@{-}[ul]_{i_{k+1}} &    \\
\bullet \ar@{-}[ur]^{j_{k}}  &&\bullet  \ar@{-}[ul]^{i_{k+1}}  \ar@{-}[ur]_{j_{k+1}}  &&\bullet  \ar@{-}[ul]_{i_{k}} \\ 
 &\ar@{-}[ul]^{i_{k+1}}  \bullet \ar@{-}[ur]^{j_k} \ar@{-}[dr]_{i_k} && \bullet \ar@{-}[ul]_{i_k} \ar@{-}[dl]^{j_k}  \ar@{-}[ur]_{j_{k+1}} \\ 
 && \bullet && \\
 && \bullet \ar@{-}[u]_{i_{k-1}= j_{k-1}} \ar@{.}[d]&&\\
 &&&& \\
}}$$

where the lattice is adapted accordingly if $i_{k+1} = j_k$ and $j_{k+1} = i_k$,  and so on.

Together with Remark~\ref{rm::lattice squares} this implies that $s_{i_1} \cdots s_{i_t} =  s_{j_1} \cdots  s_{j_t} $ and we proceed inductively to obtain $\sigma = \sigma'$.  
\end{proof}

In view of Lemma~\ref{lem::reduced expressions}, any maximal chain in $\call (M)$ gives rise to the same Coxeter element. We call this element the \emph{associated Coxeter element} of $M$ and denote it by $\sigma_M$.

The following result shows that the lattices $\call(M)$ (and therefore also the lattice  $\call(\calg)$) gives the complete set of reduced expressions of $\sigma_M$. Namely, each maximal chain corresponds to a distinct reduced expression of $\sigma_M$. More precisely, we have the following. 

\begin{theorem}\label{thm::weak order}
Let $M$ be a string module over $A=KQ/I$  and let $\sigma_M$ be the associated Coxeter element as defined above. Then we have  a lattice bijection between 
 the lattice corresponding to the interval $[e,\sigma_M]$ in the weak Bruhat order of $\frak{S}_{n+1}$ and the canonical submodule lattice  $\call (M)$.
\end{theorem}

\begin{proof}
By Lemma~\ref{lem::reduced expressions} it is enough to show that every reduced expression of $\sigma_M$ arises as a maximal chain of inclusions of canonical submodules of $M$ in $\call(M)$. Let $R(M)$ be the set  of reduced expressions of $\sigma_M$ given   by the maximal chains in $\call(M)$.  Suppose that $ s_{i_1} \cdots s_{i_n} $ is a reduced expression of $\sigma_M$ such that $ s_{i_1} \cdots s_{i_n} \notin R(M)$ and let $s_{j_1} \cdots s_{j_n} $ be a reduced expression of $\sigma_M$  in $R(M)$. Then there is a sequence of reduced expressions of $\sigma_M$  of the form $r_0 = s_{j_1} \cdots s_{j_n}, r_1, \ldots, r_k =s_{i_1} \cdots s_{i_n}$ where $ r_{l+1}$ is obtained from $r_l$ by a single commutativity relation of the  form $s_{j_{t}} s_{j_{t+1}} = s_{j_{t+1}} s_{j_{t}}$ for  $ 1 \leq t \leq n-1$.

Without loss of generality we can assume that $r_1 \notin R(M)$  and that $ r_1 = s_{j_1} \cdots  s_{j_{t+1}} s_{j_t} \cdots s_{j_n}$. This implies that in $\call(M)$ there is no mesh with edges labelled by $j_t$ and $j_{t+1}$ and that locally $\call(M)$ has the form $$\scalebox{.8}
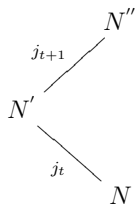
{\xymatrix{ & N^{''} \ar@{-}[dl]_{j_{t+1}} \\
 N^{'} \ar@{-}[dr]_{j_{t}} &\\
 & N
 }}$$ where $N, N^{'}, N^{''}$ are submodules of $M$. That is $j_t$ corresponds to a simple $S_{j_t}$ in the top of $N'$ but it is not in the top of $N''$ and $j_{t+1}$ corresponds to a simple $S_{j_{t+1}}$ that is not in the top of $N'$ but it is in the top of $N''$. This implies that the simples $S_{j_t}$ and $S_{j_{t+1}}$ are in consecutive radical layers in $M$ and since there is no mesh with edges $j_t$ and $j_{t+1}$, we have that  $S_{j_t}$ and $S_{j_{t+1}}$ correspond to consecutive vertices in the string of $M$. By our labelling of the string of $M$, this implies that $j_{t+1} = j_t +1$. But this contradicts $s_{j_{t}} s_{j_{t+1}} = s_{j_{t+1}} s_{j_{t}}$ and therefore $r_1 \in R(M)$. This completes the proof. 
\end{proof}

In the following we summarise the lattice bijections we have obtained.

\begin{theorem}\label{thm::LatticeBijections}
Let $M$ be a string module over $A=KQ/I$ with corresponding snake graph $\calg$ and let $\sigma_M$ be the associated Coxeter element as defined above. Then we have  lattice bijections between 
\begin{enumerate}
\item the lattice corresponding to the interval $[e,\sigma_M]$ in the weak Bruhat order of $\frak{S}_{n+1}$,
\item  the canonical submodule lattice  $\call (M)$, 
\item the lattice of perfect matchings $\call (\calg)$.
\end{enumerate}
\end{theorem}

\section{Bijections induced by skein relations}\label{Sec::NewBijection}

In \cite{CS1} crossing snake graphs  and their resolutions are defined and a bijection of the perfect matching lattices of a pair of crossing snake graphs and the perfect matching lattices of their resolutions are given. In this Section we will show that the presentation of this bijection can be given in a much simplified form. For this we begin by interpreting the notions of crossing snake graphs and their resolutions in terms of abstract string combinatorics.

\begin{definition}
(1) We say two \emph{abstract strings} $w_1$ and $w_2$ cross in  a string $m$  if $w_1 =  u_1 a m b v_1$ and $w_2 = u_2 c m d v_2$ where $u_i, u_i, m \in \calb$,  $a, d$ are arrows and $b, c$ are inverse arrows.

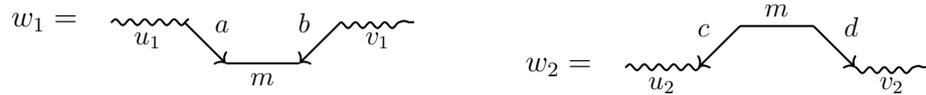
\begin{figure}[H]
\[
\resizebox{.8\textwidth}{!}{
\begin{tikzpicture}
\node at (0,0){$\begin{tikzpicture}[node distance=1cm and 1.5cm]
\coordinate[label=left:{}] (1);
\coordinate[right=1cm of 1] (13);
\coordinate[right=1cm of 13] (14);
\coordinate[below right=.8cm of 14] (15);
\coordinate[right=1cm of 15] (16);
\coordinate[above right=.8cm of 16] (17);
\coordinate[right=1cm of 17] (18);

\coordinate[right=-.5cm of 1,label=right:{$w_1=$}] (1');

\draw[thick,decorate,decoration={snake,amplitude=.4mm,segment length=2mm}] 
(13)-- node [anchor=north,scale=.9]{$u_1$} (14) 
(17)-- node [anchor=north,scale=.9]{$v_1$} (18);

\draw[thick] 
(15)-- node [anchor=north,scale=.9]{$m$} (16);

\draw[thick,->] (14)--node [anchor=south west,scale=.9]{$a$}(15);
\draw[thick,->] (17)--node [anchor=south east,scale=.9]{$b$}(16);
\end{tikzpicture}$};

\node at (7,0){$\begin{tikzpicture}[node distance=1cm and 1.5cm]
\coordinate[label=left:{}] (1);
\coordinate[right=1cm of 1] (2);
\coordinate[above right=.8cm of 2,label=right:{}] (3);
\coordinate[right=1cm of 3] (4);
\coordinate[below right=.8cm of 4] (5);
\coordinate[right=1cm of 5] (6);

\coordinate[right=-1.5cm of 1,label=right:{$w_2=$}] (1');

\draw[thick,decorate,decoration={snake,amplitude=.4mm,segment length=2mm}] 
(1)-- node [anchor=north,scale=.9]{$u_2$} (2) 
(5)-- node [anchor=north,scale=.9]{$v_2$} (6);

\draw[thick] 
(3)-- node [anchor=south,scale=.9]{$m$} (4);

\draw[thick,->] (3)--node [anchor=south east,scale=.9]{$c$}(2); 
\draw[thick,->] (4)--node [anchor=south west,scale=.9]{$d$}(5);
\end{tikzpicture}$};
\end{tikzpicture}
}
\]
\end{figure}

The \emph{resolution} of the  crossing of $w_1$ and $w_2$ in $m$ are the following abstract strings 
\begin{itemize}
\item[-] $w_3 = u_1 a m d v_2$
\item[-] $w_4 = u_2 c m b v_1$
\item[-] $w_5 = \left\{ \begin{array}{ll}
u_1 e  u_2^{-1} & \mbox{ if } u_1, u_2 \neq \emptyset \mbox{ and where } e = \leftarrow \\
 u'_2  & \mbox{ if } u_1 =  \emptyset \mbox{ and } u_2 \neq \emptyset  \\
 u'_1  & \mbox{ if } u_2 = \emptyset \mbox{ and } u_1 \neq \emptyset  \\
 \end{array} \right. $\\
where  $u'_2$ is such that $u_2 = u'_2 r''$ where $r'' $ is a maximal sequence of direct arrows and $u'_1$ is such that $u_1 = u'_1 r'$ where $r' $ is a maximal sequence of inverse arrows. 
 \item[-] $w_6  = \left\{ \begin{array}{ll}
v_1^{-1} f  v_2 & \mbox{ if } v_1, v_2 \neq \emptyset \mbox{ and where } f = \leftarrow \\
 v'_2  & \mbox{ if } v_1 =  \emptyset \mbox{ and } v_2 \neq \emptyset  \\
 v'_1  & \mbox{ if } v_2 = \emptyset \mbox{ and } v_1 \neq \emptyset  \\
 \end{array} \right. $ \\
where $v'_1$ is such that $v_2 = s' v'_2$ where $s' $ is a maximal sequence of inverse arrows and $v'_2$ is such that $v_2 = v'_2 r''$ where $r'' $ is a maximal sequence of direct arrows. 
\end{itemize} 
The resolution of a crossing of $w_1$ and $w_2$ is  illustrated in Figure~(\ref{w3-w6}).

\begin{figure}
\[ \resizebox{0.7\textwidth}{!}{
\begin{tikzpicture}
\node at (0,0){$\begin{tikzpicture}[node distance=1cm and 1.5cm]
\coordinate[label=left:{}] (1);
\coordinate[above right=1cm of 1] (7);
\coordinate[right=1cm of 7] (8);
\coordinate[below right=.8cm of 8] (9);
\coordinate[right=1cm of 9] (10);
\coordinate[below right=.8cm of 10] (11);
\coordinate[right=1cm of 11] (12);

\coordinate[right=-.5cm of 1,label=right:{$w_3=$}] (1');

\draw[thick,decorate,decoration={snake,amplitude=.4mm,segment length=2mm}] 
(7)-- node [anchor=south,scale=.9]{$u_1$} (8) 
(11)-- node [anchor=north,scale=.9]{$w_2$} (12) ;

\draw[thick] 
(9)-- node [anchor=south,scale=.9]{$m$} (10);

\draw[thick,->] (8)--node [anchor=south west,scale=.9]{$a$}(9);
\draw[thick,->] (10)--node [anchor=south west,scale=.9]{$d$}(11);
\end{tikzpicture}$};

\node at (7,0){$\begin{tikzpicture}[node distance=1cm and 1.5cm]
\coordinate[label=left:{}] (1);
\coordinate[below right=1cm of 1] (7');
\coordinate[right=1cm of 7'] (8');
\coordinate[above right=.8cm of 8'] (9');
\coordinate[right=1cm of 9'] (10');
\coordinate[above right=.8cm of 10'] (11');
\coordinate[right=1cm of 11'] (12');

\coordinate[right=-.5cm of 1,label=right:{$w_4=$}] (1');

\draw[thick,decorate,decoration={snake,amplitude=.4mm,segment length=2mm}] 
(7')-- node [anchor=north,scale=.9]{$u_2$} (8') 
(11')-- node [anchor=north,scale=.9]{$v_1$} (12') 
;

\draw[thick] 
(9')-- node [anchor=north,scale=.9]{$m$} (10')
;

\draw[thick,<-] (8')--node [anchor=north west,scale=.9]{$c$}(9');
\draw[thick,<-] (10')--node [anchor=north west,scale=.9]{$b$}(11');
\end{tikzpicture}
$};

\node at (-.5,-3){$\begin{tikzpicture}[node distance=1cm and 1.5cm]
\coordinate[label=left:{}] (1);
\coordinate[below right=1cm of 1] (7');
\coordinate[right=1cm of 7'] (8');
\coordinate[above right=.8cm of 8'] (9');
\coordinate[right=1cm of 9'] (10');
\coordinate[above right=.8cm of 10'] (11');
\coordinate[right=1cm of 11'] (12');

\coordinate[right=-.5cm of 1,label=right:{$w_5=$}] (1');

\draw[thick,decorate,decoration={snake,amplitude=.4mm,segment length=2mm}] 
(7')-- node [anchor=north,scale=.9]{$u_1$} (8') 
(9')-- node [anchor=south,scale=.9]{$u_2^{-1}$} (10');

\draw[thick,<-] (8')--node [anchor=north west,scale=.9]{$e$}(9');
\end{tikzpicture}
$};

\node at (6.5,-3){$\begin{tikzpicture}[node distance=1cm and 1.5cm]
\coordinate[label=left:{}] (1);
\coordinate[below right=1cm of 1] (7');
\coordinate[right=1cm of 7'] (8');
\coordinate[above right=.8cm of 8'] (9');
\coordinate[right=1cm of 9'] (10');
\coordinate[above right=.8cm of 10'] (11');
\coordinate[right=1cm of 11'] (12');

\coordinate[right=-.5cm of 1,label=right:{$w_6=$}] (1');

\draw[thick,decorate,decoration={snake,amplitude=.4mm,segment length=2mm}] 
(7')-- node [anchor=north,scale=.9]{$v_1^{-1}$} (8') 
(9')-- node [anchor=south,scale=.9]{$v_2$} (10');

\draw[thick,<-] (8')--node [anchor=north west,scale=.9]{$f$}(9');
\end{tikzpicture}
$};
\end{tikzpicture}
}\]
\caption{The strings $w_3, w_4, w_5$ and $w_6$.}\label{w3-w6}
\end{figure}
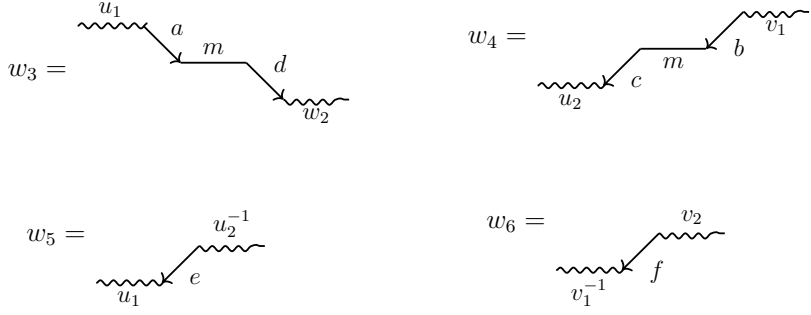

(2) In accordance with the terminology in \cite{CS1}, for two strings $w_1 = u_1  a  v_1$ and $w_2$, a \emph{grafting} of $w_2$ onto $w_1$ at $a$ is given by the string $u_1ew_2$ where $e = a^{-1}$ if $a \in \{ \to,\leftarrow\}$ and $e \in\{ \to,\leftarrow\}$ if $a=\emptyset$.

The \emph{resolution} of a grafting by a fixed $e\in\{\to,\leftarrow\}$ is given by the strings
\begin{itemize}
\item[-] $w_3 = u_1  e w_2$  
\item[-] $w_4 = v'_1$ where $v_1 = r v'_1$ and where is $r  = a_1 \ldots a_n$ is a maximal prefix of $v_2$  such that $a_i = a$  for $i=1,\dots,n$
\item[-] $w_5 = u'_1$ where $u_1 = u'_1 r'$ where $r'   = a'_1 \ldots a'_m$ is a maximal suffix of $v_1$  such that $a'_i = a$  for $i=1,\dots,m$
\item[-] $w_6  = \left\{ \begin{array}{ll} 
 v_1^{-1} a w_2  & \mbox{ if } a \neq  \emptyset \\ 
w_2' & \mbox{ if } a=  \emptyset \\
\end{array} \right.$, where $w_2 = r'' w'_2$ where $r''   = a''_1 \ldots a''_k$ is a maximal prefix of $w_2$  such that $a''_i = a$  for $i=1,\dots,k$.
\end{itemize}

\end{definition}

\begin{figure}
\[
\resizebox{1\textwidth}{!}{
\begin{tikzpicture}

\node[scale=.5] at (0,0){$\begin{tikzpicture}[node distance=1cm and 1.5cm]
\coordinate[label=left:{}] (1);
\coordinate[right=1cm of 1] (2);
\coordinate[below right=.8cm of 2] (3);
\coordinate[right=1cm of 3] (4);

\coordinate[right=-1.5cm of 1,label=right:{$w_1=$}] (1'');

\draw[thick,decorate,decoration={snake,amplitude=.4mm,segment length=2mm}] 
(1)-- node [anchor=north,scale=.9]{$u_1$} (2) 
(3)-- node [anchor=north,scale=.9]{$v_1$} (4);

\draw[thick,->] (2)--node [anchor=south west,scale=.9]{$a$}(3);

\coordinate[right=2.5cm of 1,label=right:{$=$}] (1'');

\coordinate[above right=3cm of 2] (5);
\coordinate[right=.8cm of 5] (5');
\coordinate[below right=.6cm of 5'] (6);
\coordinate[below right=.6cm of 6] (7);
\coordinate[below right=.6cm of 7] (8);
\coordinate[below right=.6cm of 8] (9);
\coordinate[below right=.6cm of 9] (10);
\coordinate[below right=.6cm of 10] (11);
\coordinate[below right=.6cm of 11] (12);
\coordinate[right=.8cm of 12] (13);

\draw[thick,decorate,decoration={snake,amplitude=.4mm,segment length=2mm}] 
(5)-- node [anchor=south,scale=.9]{$u'_1$} (5') 
(12)-- node [anchor=north,scale=.9]{$v'_1$} (13);

\draw[->] (5')--node [anchor=south west,scale=.9]{$a'_1$}(6);
\draw[dotted] (6)--node [anchor=south west,scale=.9]{$a'_1$}(7);
\draw[->] (7)--node [anchor=south west,scale=.9]{$a'_m$}(8);
\draw[thick,->] (8)--node [anchor=south west,scale=.9]{$a$}(9);
\draw[->] (9)--node [anchor=south west,scale=.9]{$a_1$}(10);
\draw[dotted] (10)--node [anchor=south west,scale=.9]{$a'_1$}(11);
\draw[->] (11)--node [anchor=south west,scale=.9]{$a_n$}(12);


\coordinate[below=3cm of 1] (1');
\coordinate[right=1cm of 1'] (2');

\coordinate[right=-1.5cm of 1',label=right:{$w_2=$}] (1''');
\coordinate[right=.7cm of 2',label=left:{$=$}] (3');
\coordinate[above right=.6cm of 3'] (4');
\coordinate[below right=.6cm of 4'] (5');
\coordinate[below right=.6cm of 5'] (6');
\coordinate[below right=.6cm of 6'] (7');
\coordinate[right=.8cm of 7'] (8');

\draw[thick,decorate,decoration={snake,amplitude=.4mm,segment length=2mm}] 
(1')-- node [anchor=north,scale=.7]{} (2') ;
\draw[thick,decorate,decoration={snake,amplitude=.3mm,segment length=2mm}] 
(7')-- node [anchor=north,scale=.7]{$w'_2$} (8') ;

\draw[->] (4')--node [anchor=south west,scale=.7]{$a''_1$}(5');
\draw[dotted] (5')--node [anchor=south west,scale=.7]{}(6');
\draw[->] (6')--node [anchor=south west,scale=.7]{$a''_k$}(7');
\end{tikzpicture}$};

\node[scale=.5] at (6,0){$\begin{tikzpicture}[node distance=1cm and 1.5cm]
\coordinate[label=left:{}] (1);
\coordinate[right=1cm of 1] (2);
\coordinate[above right=.8cm of 2] (3);
\coordinate[right=1cm of 3] (4);

\coordinate[right=-1.5cm of 1,label=right:{$w_1=$}] (1'');

\draw[thick,decorate,decoration={snake,amplitude=.4mm,segment length=2mm}] 
(1)-- node [anchor=north,scale=.9]{$u_1$} (2) 
(3)-- node [anchor=south,scale=.9]{$v_1$} (4);

\draw[thick,->] (3)--node [anchor=north west,scale=.9]{$a$}(2);

\coordinate[right=2.5cm of 1,label=right:{$=$}] (1'');

\coordinate[below right=2cm of 3] (5);
\coordinate[right=.8cm of 5] (5');
\coordinate[above right=.6cm of 5'] (6);
\coordinate[above right=.6cm of 6] (7);
\coordinate[above right=.6cm of 7] (8);
\coordinate[above right=.6cm of 8] (9);
\coordinate[above right=.6cm of 9] (10);
\coordinate[above right=.6cm of 10] (11);
\coordinate[above right=.6cm of 11] (12);
\coordinate[right=.8cm of 12] (13);

\draw[thick,decorate,decoration={snake,amplitude=.4mm,segment length=2mm}] 
(5)-- node [anchor=north,scale=.9]{$u'_1$} (5') 
(12)-- node [anchor=south,scale=.9]{$v'_1$} (13);

\draw[<-] (5')--node [anchor=north west,scale=.9]{$a'_1$}(6);
\draw[dotted] (6)--node [anchor=north west,scale=.9]{$a'_1$}(7);
\draw[<-] (7)--node [anchor=north west,scale=.9]{$a'_m$}(8);
\draw[thick,<-] (8)--node [anchor=north west,scale=.9]{$a$}(9);
\draw[<-] (9)--node [anchor=north west,scale=.9]{$a_1$}(10);
\draw[dotted] (10)--node [anchor=north west,scale=.9]{$a'_1$}(11);
\draw[<-] (11)--node [anchor=north west,scale=.9]{$a_n$}(12);

\coordinate[below=4cm of 1] (1');
\coordinate[right=1cm of 1'] (2');

\coordinate[right=-1.5cm of 1',label=right:{$w_2=$}] (1''');

\coordinate[right=.7cm of 2',label=left:{$=$}] (3');
\coordinate[below right=.6cm of 3'] (4');
\coordinate[above right=.6cm of 4'] (5');
\coordinate[above right=.6cm of 5'] (6');
\coordinate[above right=.6cm of 6'] (7');
\coordinate[right=.8cm of 7'] (8');

\draw[thick,decorate,decoration={snake,amplitude=.4mm,segment length=2mm}] 
(1')-- node [anchor=north,scale=.7]{} (2') ;
\draw[thick,decorate,decoration={snake,amplitude=.3mm,segment length=2mm}] 
(7')-- node [anchor=south,scale=.7]{$w'_2$} (8') ;

\draw[<-] (4')--node [anchor=south east,scale=.7]{$a''_1$}(5');
\draw[dotted] (5')--node [anchor=south east,scale=.7]{}(6');
\draw[<-] (6')--node [anchor=south east,scale=.7]{$a''_k$}(7');

\end{tikzpicture}$};

\draw[dotted] (3,2)--(3,-6);


\node[scale=.5] at (0,-4){$\begin{tikzpicture}[node distance=1cm and 1.5cm]
\coordinate[label=left:{}] (1);
\coordinate[right=1cm of 1] (2);
\coordinate[above right=.8cm of 2] (3);
\coordinate[right=1cm of 3] (4);

\coordinate[right=-1.5cm of 1,label=right:{$w_3=$}] (1'');

\draw[thick,decorate,decoration={snake,amplitude=.4mm,segment length=2mm}] 
(1)-- node [anchor=north,scale=.9]{$u_1$} (2) 
(3)-- node [anchor=south,scale=.9]{$w_2$} (4);

\draw[thick,->] (3)--node [anchor=north west,scale=.9]{$e$}(2);

\coordinate[below=1.5cm of 1] (1');
\coordinate[right=1cm of 1'] (2');
\coordinate[right=-1.5cm of 1',label=right:{$w_4=$}] (1'');
\draw[thick,decorate,decoration={snake,amplitude=.4mm,segment length=2mm}] 
(1')-- node [anchor=north,scale=.9]{$v'_1$} (2') ;

\coordinate[below=3cm of 1] (1'');
\coordinate[right=1cm of 1''] (2'');
\coordinate[right=-1.5cm of 1'',label=right:{$w_5=$}] (1''');
\draw[thick,decorate,decoration={snake,amplitude=.4mm,segment length=2mm}] 
(1'')-- node [anchor=north,scale=.9]{$u'_1$} (2'') ;

\coordinate[below=4.5cm of 1] (16);
\coordinate[right=1cm of 16] (26);
\coordinate[below right=.8cm of 26] (36);
\coordinate[right=.8cm of 36] (46);

\coordinate[right=-1.5cm of 16,label=right:{$w_6=$}] (16');
\draw[thick,decorate,decoration={snake,amplitude=.4mm,segment length=2mm}] 
(16)-- node [anchor=south,scale=.9]{$v^{-1}_1$} (26) ;

\draw[->] (26)--node [anchor=north east,scale=.7]{$c$}(36);

\draw[thick,decorate,decoration={snake,amplitude=.4mm,segment length=2mm}] 
(36)-- node [anchor=north,scale=.9]{$w_2$} (46) ;

\coordinate[right=.8cm of 46, label=above: OR] (56);
\coordinate[right=2cm of 56, label=above:{$w_6=w'_2$}] (66);

\end{tikzpicture}$};


\node[scale=.5] at (6,-4){$\begin{tikzpicture}[node distance=1cm and 1.5cm]
\coordinate[label=left:{}] (1);
\coordinate[right=1cm of 1] (2);
\coordinate[below right=.8cm of 2] (3);
\coordinate[right=1cm of 3] (4);

\coordinate[right=-1.5cm of 1,label=right:{$w_3=$}] (1'');

\draw[thick,decorate,decoration={snake,amplitude=.4mm,segment length=2mm}] 
(1)-- node [anchor=south,scale=.9]{$u_1$} (2) 
(3)-- node [anchor=north,scale=.9]{$w_2$} (4);

\draw[thick,->] (2)--node [anchor=south west,scale=.9]{$e$}(3);

\coordinate[below=1.5cm of 1] (1');
\coordinate[right=1cm of 1'] (2');
\coordinate[right=-1.5cm of 1',label=right:{$w_4=$}] (1'');
\draw[thick,decorate,decoration={snake,amplitude=.4mm,segment length=2mm}] 
(1')-- node [anchor=north,scale=.9]{$v'_1$} (2') ;

\coordinate[below=3cm of 1] (1'');
\coordinate[right=1cm of 1''] (2'');
\coordinate[right=-1.5cm of 1'',label=right:{$w_5=$}] (1''');
\draw[thick,decorate,decoration={snake,amplitude=.4mm,segment length=2mm}] 
(1'')-- node [anchor=north,scale=.9]{$u'_1$} (2'') ;

\coordinate[below=5.5cm of 1] (16);
\coordinate[right=1cm of 16] (26);
\coordinate[above right=.8cm of 26] (36);
\coordinate[right=.8cm of 36] (46);

\coordinate[right=-1.5cm of 16,label=right:{$w_6=$}] (16');
\draw[thick,decorate,decoration={snake,amplitude=.4mm,segment length=2mm}] 
(16)-- node [anchor=north,scale=.9]{$v^{-1}_1$} (26) ;

\draw[<-] (26)--node [anchor=south east,scale=.7]{$c$}(36);

\draw[thick,decorate,decoration={snake,amplitude=.4mm,segment length=2mm}] 
(36)-- node [anchor=south,scale=.9]{$w_2$} (46) ;

\coordinate[right=.8cm of 46, label=below: OR] (56);
\coordinate[right=2cm of 56, label=below:{$w_6=w'_2$}] (66);

\end{tikzpicture}$};

\end{tikzpicture}
}
\]
\caption{Illustration of a grafting of $w_2$ onto $w_1$ at $a$ } \label{Fig:CrossString}
\end{figure}
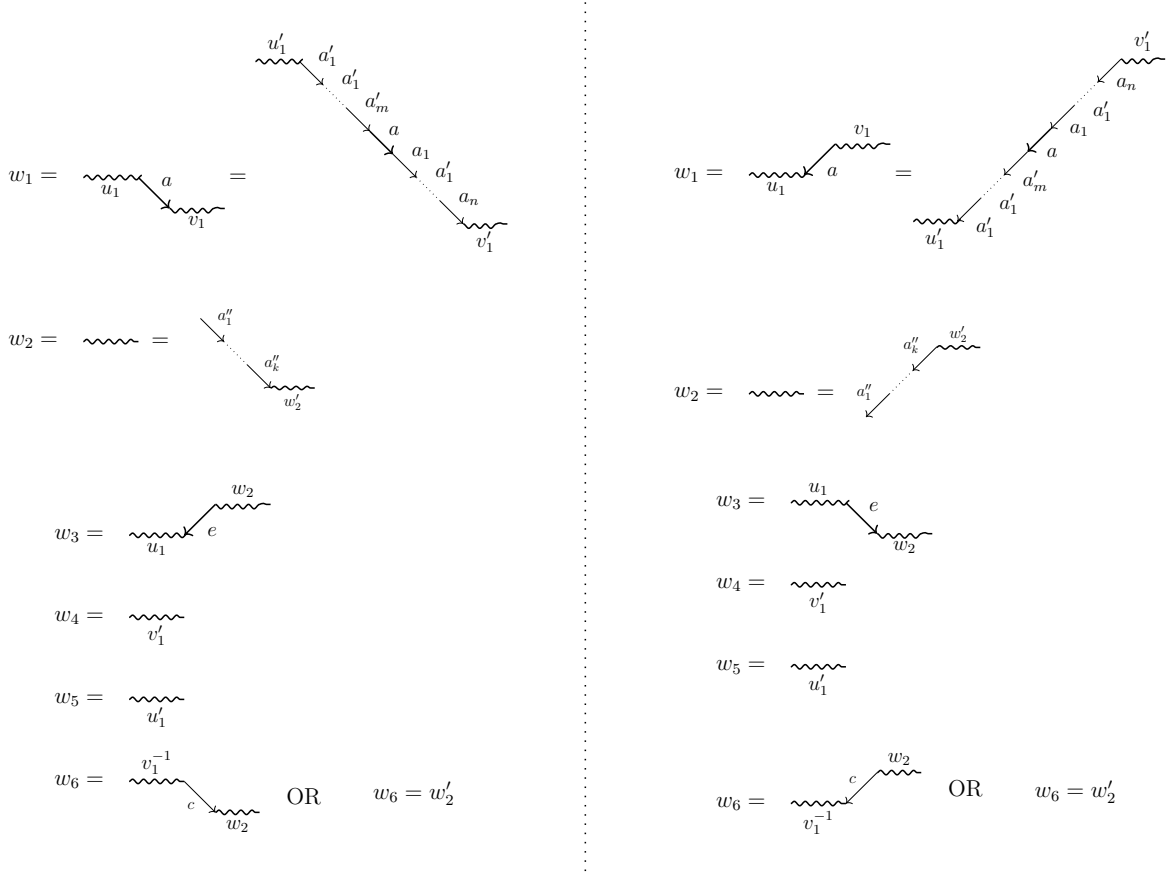

\begin{definition} We say that snake graphs $\calg$ and $\calg'$ \emph{cross} if their associated abstract strings cross and their \emph{resolution} is given by the snake graphs $\calg_3, \calg_4, \calg_5$ and $ \calg_6$ associated to the resolution of the crossing strings.
\end{definition}

Let $\calg_1, \ldots, \calg_6$ be the snake graphs associated to $w_1, \ldots, w_6$, respectively.

\begin{example}
	Let $w=1\rightarrow 2\rightarrow 3\rightarrow 4 \leftarrow 5\rightarrow 6$ and $w_2= 7 \leftarrow 3 \rightarrow 4 \rightarrow 8.$ Then $w_1$ crosses $w_2$ in $m=3 \rightarrow 4$ and its resolution is given by $w_3=1\rightarrow 2 \rightarrow 3 \rightarrow 4 \rightarrow 8$, $w_4=7 \leftarrow 3 \rightarrow 4 \leftarrow 5\rightarrow 6$, $w_5=1\rightarrow 2 \leftarrow 7$ and $w_6= 6 \leftarrow 5 \leftarrow 8$.
The corresponding snake graphs of this resolution are as follows:

\begin{figure}[!htbp]
\begin{tikzpicture}

\node at (0,0){$\begin{tikzpicture}[scale=.4]
\draw (0,0)--(2,0)--(2,2)--(3,2)--(3,1)--(0,1)--(0,0) (1,0)--(1,2)--(2,2)
(3,1)--(5,1)--(5,2)--(3,2) 
(4,1)--(4,2);
\node[scale=.8] at (.5,.5){$1$};
\node[scale=.8] at (1.5,.5){$2$};
\node[scale=.8] at (1.5,1.5){$3$};
\node[scale=.8] at (2.5,1.5){$4$};
\node[scale=.8] at (3.5,1.5){$5$};
\node[scale=.8] at (4.5,1.5){$6$};
\end{tikzpicture}$};

\node at (2.5,0){$\begin{tikzpicture}[scale=.4]
\draw (0,1)--(3,1)--(3,3)--(2,3)--(2,1)
(0,1)--(0,2)--(3,2)
(1,1)--(1,2)
;
\node[scale=.8] at (.5,1.5){$7$};
\node[scale=.8] at (1.5,1.5){$3$};
\node[scale=.8] at (2.5,1.5){$4$};
\node[scale=.8] at (2.5,2.5){$8$};
\end{tikzpicture}
$};

\node at (5,0){$\begin{tikzpicture}[scale=.4]
\draw (0,0)--(0,1)--(1,1)
(0,0)--(2,0)--(2,1)
(3,2)--(1,2)--(1,0)
(1,1)--(3,1)--(3,3)--(2,3)--(2,1)
;
\node[scale=.8] at (.5,.5){$1$};
\node[scale=.8] at (1.5,.5){$2$};
\node[scale=.8] at (1.5,1.5){$3$};
\node[scale=.8] at (2.5,1.5){$4$};
\node[scale=.8] at (2.5,2.5){$8$};
\end{tikzpicture}
$};

\node at (7.5,0){$\begin{tikzpicture}[scale=.4]
\draw (0,0)--(5,0)--(5,1)--(0,1)--(0,0)
(1,0)--(1,1) (2,0)--(2,1) (3,0)--(3,1) (4,0)--(4,1)
;
\node[scale=.8] at (.5,.5){$7$};
\node[scale=.8] at (1.5,.5){$3$};
\node[scale=.8] at (2.5,.5){$4$};
\node[scale=.8] at (3.5,.5){$5$};
\node[scale=.8] at (4.5,.5){$6$};
\end{tikzpicture}
$};

\node at (10,0){$\begin{tikzpicture}[scale=.4]
\draw (1,1)--(3,1)--(3,3)--(2,3)--(2,1)
(1,1)--(1,2)--(3,2)
;
\node[scale=.8] at (1.5,1.5){$1$};
\node[scale=.8] at (2.5,1.5){$2$};
\node[scale=.8] at (2.5,2.5){$7$};
\end{tikzpicture}$};

\node at (12,0){$\begin{tikzpicture}[scale=.4]
\draw (1,1)--(3,1)--(3,3)--(2,3)--(2,1)
(1,1)--(1,2)--(3,2)
;
\node[scale=.8] at (1.5,1.5){$6$};
\node[scale=.8] at (2.5,1.5){$5$};
\node[scale=.8] at (2.5,2.5){$8$};
\end{tikzpicture}
$};

\node at (0,-1){$\calg_1$};
\node at (2.5,-1){$\calg_2$};
\node at (5,-1){$\calg_3$};
\node at (7.5,-1){$\calg_4$};
\node at (10,-1){$\calg_5$};
\node at (12,-1){$\calg_6$};
\end{tikzpicture}
\end{figure}

\end{example}

We now reformulate Theorem~3.1 in \cite{CS1} in a concise single statement by collecting the effect of Figures~{6, 7, 8, 9, 10, and 11 in \cite{CS1}} in the bijection into two simple conditions given below.  Let $\calg_1$ and $\calg_2$  be snake graphs crossing in a snake graph $\calg$  such that $\calg = \calg_1 [s,t] = \calg_2 [s', t']$, and let $P_1$ and $P_2$ be the perfect matchings of $\calg_1$ and $\calg_2$, respectively. Denote by $T_i ,B_i, R_i, L_i$ the top, bottom, right and left edges of a tile $G_i$ of $\calg$, respectively. We define condition

$(*)$ to be the existence of $i \in [s,t]$ such that 
\[ 
P_1 [i] \cup P_2 [i] \supset \{B_i, L_i \}  \mbox{ or } P_1 [i] \cap P_2 [i] \subset \{B_i, L_i \} 
\]
 and 
$(**)$ to be the existence of $ s \leq i < t$ or $i = s-1$ or $i = s'-1$ such that 
\[ 
P_1 [i] \cup P_2 [i] \supset \{T_i, R_i \}
\]
where the notation $P_j[i]$ denotes the matched edges in $P_j$ of the  $i$-th tile in $\calg_1$.

We obtain the following reformulation of Theorem~3.1 in \cite{CS1}.

\begin{theorem}\label{thm::CS bijection}
Let $\calg_1, \calg_2$ be two snake graphs, and  $\calg_3 \ldots, \calg_6$ be the snake graphs obtained by resolving either an overlap or grafting of the associated strings. 
The following is a bijective map 
\[
\Phi : \match \calg_1 \times \match \calg_2 \longrightarrow \match \calg_3 \times \match \calg_4 \sqcup  \match \calg_5 \times \match \calg_6
\]
which in the case of  a grafting  is given by 
\[ 
\Phi(P_1, P_2) = \left\{ \begin{array}{ll}
(P_1 \vert_{\calg_3} \cup P_2 \vert_{\calg_3}, \, P_1 \vert_{\calg_4} \cup P_2 \vert_{\calg_4}) & \mbox{ if } e \in P_1 \cup P_2 \\ 
& \\
(P_1 \vert_{\calg_5} \cup P_2 \vert_{\calg_5}, \, P_1 \vert_{\calg_6} \cup P_2 \vert_{\calg_6}) & \mbox{ if } e \notin P_1 \cup P_2 \\ 
\end{array} \right. 
\]
and in the case of $\calg_1$ and $\calg_2$ crossing in $\calg$ such that $\calg = \calg_1 [s,t] = \calg_2 [s', t']$ it is given by 
\[
\Phi(P_1, P_2) = \left\{ \begin{array}{ll}
(P_3, P_4) & \mbox{ if } (*) \mbox{ or } (**) \mbox{ holds } \\
& \\
(P_1 \mid_{\calg_5} \cup P_2 \mid_{\calg_5}, \, P_1 \mid_{\calg_6}  \cup P_2 \mid_{\calg_6}) & \mbox{ otherwise }
\end{array} \right. 
\]
where if $i$ is the minimal integer satisfying $(*)$ or $(**)$ then 
\[
(P_3, P_4)  = \left\{ \begin{array}{ll}
(P_1 [1, i-1 [ \; \cup \; P_2 [i,d], \, P_2 [1, i-1] \; \cup \; P_1 [i, d'] ) & \mbox{ if $(*)$ holds} \\
& \\
 (P_1 [1, i+1 [ \; \cup \; P_2 ]i+1,d], \, P_2 [1, i+1[ \; \cup \; P_1 ]i+1, d'] ) & \mbox{ if $(**)$ holds} \\
\end{array} \right. 
\]
 where $P[i,j[$ (resp. $P]i,j]$) is the restriction of the perfect matching $P$ on the tiles $T_i,\dots,T_j$ of $\calg$ excluding the top and right (resp. bottom and left) edges of $T_j$ (resp. $T_i$).
\end{theorem}

The proof is a straightforward case by case comparison with the cases detailed in the proof of Theorems 3.2 in \cite{CS1}.

\begin{remark}
 Theorem~\ref{thm::CS bijection} can be extended in the same way to include the bijections associated to self-crossing snake graphs and band graphs as given in \cite[Theorem 4.4]{CS2} and \cite[Theorem 4.8]{CS3}, respectively. 
\end{remark}

Using the symmetric difference of the minimal perfect matching with the perfect matchings in Theorem~\ref{thm::CS bijection}, we obtain the following Corollary.
 
\begin{corollary} Let $\calg_1, \ldots, \calg_6$ be the snake graphs in Theorem~\ref{thm::CS bijection} and let $M_i = M(\calg_i)$, for $i \in \{ 1, \ldots, 6\}$ be the associated string modules. Then  
 we obtain a bijection on the submodule lattices 

$$\call (M_1) \times \call (M_2)  \stackrel{\Phi}{\simeq} \call(M_3)  \times \call(M_4) \sqcup \call(M_5)  \times \call(M_6).  $$
\end{corollary}

Combining the results of the previous  sections with the results in  \cite{S} (see also \cite{CaSc} and \cite{CPS}),  we obtain the following. 

\begin{proposition}\label{prop::extension}
With the notation above,  for submodules $N_1$ and $N_2$ of $M_1$ and $M_2$ respectively, the map $\Phi$ gives rise to extensions of $N_1$ by $N_2$. Namely  if $\Phi (N_1, N_2) = (N_3, N_4) \subset \call(M_3) \times \call(M_4) $  then there is an extension 
$$ 0 \to N_2 \to N_3 \oplus N_4 \to N_1 \to 0.$$
\end{proposition}

We note that the extension in Proposition~\ref{prop::extension} might be split and that  there usually are many other extensions of $N_1$ by $N_2$ not captured by $\Phi$.

\end{document}